\documentclass{amsart}

\usepackage{amssymb, latexsym, amsthm}

    \usepackage{amsmath}
    \usepackage{array}
    \usepackage{ifpdf}
    \ifpdf
       \usepackage{hyperref}
    \else
        \newcommand\texorpdfstring[2]{{#1}}
    \fi

    \newtheorem{thm}{Theorem}[section]
    \newtheorem{cor}[thm]{Corollary}
    \newtheorem{prop}[thm]{Proposition}
    \newtheorem{lem}[thm]{Lemma}
    \newtheorem{dfn}[thm]{Definition}

    \newtheorem{rem}[thm]{Remark}
    \newtheorem{exm}[thm]{Example}
\newtheorem*{thm*}{Theorem}

    \newcommand\T[2][]{{\operatorname{id}_{#1}\!\left({#2}\right)}}

    \providecommand{\tr}{\mathop{\rm tr}\nolimits}
    \providecommand{\str}{\mathop{\rm str}\nolimits}
    \providecommand{\estr}{\mathfrak{str}}
    \providecommand{\Tr}{\mathop{\rm Tr}\nolimits}
    \providecommand{\sTr}{\mathop{\rm sTr}\nolimits}
    \providecommand{\esTr}{\mathfrak{sTr}}

    \providecommand{\Cent}{\mathop{\rm Cent}\nolimits}
    \providecommand{\sCent}{\mathop{\rm sCent}\nolimits}
    \providecommand{\esCent}{\mathfrak{sCent}}

    \providecommand{\ef}{\mathfrak{f}}
    \providecommand{\eF}{\mathfrak{F}}

    \providecommand{\eA}{\mathfrak{A}}
    \providecommand{\eB}{\mathfrak{B}}
    \providecommand{\eC}{\mathfrak{C}}
    \providecommand{\eL}{\mathfrak{L}}

    \newcommand\sgn[1]{{\mathop{\rm sgn}\left({#1}\right)}}
    \newcommand\esgn[2]{{\mathfrak{sgn}_{#1\!}\left({#2}\right)}}

    \providecommand{\chr}{\mathop{\rm char}\nolimits}
    
    \providecommand{\Tor}{\mathop{\rm Tor}\nolimits}

    \def\F{\mathbb{F}}
    \def\Z{\mathbb{Z}}
    \def\N{\mathbb{N}}
    \def\freeG{S}
    \def\G{\mathfrak{G}}
    \def\efreeG{\mathfrak{S}}
    
    \newcommand\M[1][n]{{\operatorname{M}_{#1}}}
    \def\e{\varepsilon}
    \def\l{\ell}

    \def\co{{\,:\,}}
    \def\suchthat{\,|\,}

    \newcommand\sg[1]{{\left<{#1}\right>}}
    \newcommand\ideal[1]{{\left<{#1}\right>}}
    \def\({\left(}
    \def\){\right)}
    \newcommand\super{{$\Sigma$-super}} 

    \newcommand\set[1]{{\left\{{#1}\right\}}}
    \def\ra{{\rightarrow}}
    \def\s{{\sigma}}
    \newcommand\eq[1]{{(\ref{#1})}}

    \def\normali{\triangleleft}

    \newcommand\Eqref[1]{{\eqref{#1}}}
    \newcommand\Exref[1]{{Example~\ref{#1}}}
    
    \newcommand\Tref[1]{{Theorem~\ref{#1}}}
    \newcommand\Lref[1]{{Lemma~\ref{#1}}}
    \newcommand\Dref[1]{{Definition~\ref{#1}}}
    \newcommand\Pref[1]{{Proposition~\ref{#1}}}
    \newcommand\Cref[1]{{Corollary~\ref{#1}}}
    
    \newcommand\Rref[1]{{Remark~\ref{#1}}}

    \newcommand\len{{\operatorname{len}}}

\newcommand\INTROTHM[2]{{\begin{thm*}[#1] {#2} \end{thm*}}}

    \def\ThmTraceEquiv{{Suppose that $2$ is invertible in $C$. Let $A$ be some $C$-algebra with trace $\tr$. Let $\estr$ be the associated \super{}trace of $A\otimes_C \efreeG$, and in a similar manner, associate a supertrace $\str$ to $A\otimes_C \freeG$, where $\freeG$ is the free supercommutative algebra.
    Then the supertrace identities of $A\otimes_C \freeG$ are the same as the \super{}trace identities of $A\otimes_C\efreeG$, with $\esTr$ replaced by $\sTr$}}

\begin{document}
        \title[The Generalized Grassmann Algebra]%
            {The Grassmann algebra in arbitrary characteristic and generalized sign
}
        \author{Gal Dor, Alexei Kanel-Belov and Uzi Vishne}

\renewcommand{\subjclassname}{
      \textup{2010} Mathematics Subject Classification}

        \email{dorgal111@gmail.com, beloval@cs.biu.ac.il, vishne@math.biu.ac.il}
\address{
          Department of Mathematics, Bar-Ilan University, 52900 Ramat-Gan, Israel
}

        \thanks{This work was supported by BSF grant \#2010/149 and ISF grant \#1207/12.}
        \date{\today}


        \maketitle

        \begin{abstract}
            We define a generalization $\G$ of the Grassmann algebra $G$ which is well-behaved over arbitrary commutative
            rings $C$, even when $2$ is not invertible. In particular, this enables us to define a notion of superalgebras that does not become
            degenerate in such a setting. Using this construction we are able to provide a basis of the non-graded multilinear identities of the free
            superalgebra with supertrace, valid over any ring.

            We also show that all identities of $\G$ follow from the Grassmann identity, and explicitly give its co-modules, which turn
            out to be generalizations of the sign representation. In particular, we show that the co-module is a free $C$-module of rank
            $2^{n-1}$.
        \end{abstract}

    \section{Introduction and Notation}
    \numberwithin{thm}{section}

    Algebras are associative, but not necessarily unital.
    The base ring $C$ will always be commutative and unital. We will assume nothing about the characteristic of $C$, except where explicitly stated.

    Let $A$ be an algebra over $C$, and let $C\sg{X}$ be the free (associative) algebra over a countable infinite alphabet $X$. A polynomial $f(x_1,\dots,x_n)\in C\sg{X}$ is an \emph{identity} of $A$ if for
    all substitutions $a_1,\dots,a_n\in A$, we have that $f(a_1,\dots,a_n)=0$. We let:
    \[
        \T{A}=\{f\in C\sg{X} \suchthat\text{$f$ is an identity of A}\}.
    \]
    An algebra satisfying some non-zero identity with at least one invertible coefficient is called a {\emph{PI-algebra}}.

    Obviously, $\T{A}$ is an ideal of $C\sg{X}$, which is invariant under substitutions. For any ring $R$, a \emph{T-ideal} is an ideal $I \normali R$
    such that $\tau(I)\subseteq I$ for every
    endomorphism $\tau$ of $R$. We will implicitly assume throughout that all T-ideals
    are T-ideals of $C\sg{X}$.
    With this terminology, $\T{A}$ is a T-ideal for every algebra $A$.

    Given that an algebra $A$ over an infinite field $C=\F$ satisfies
    an identity $f$, it is always possible to break $f$ down into its
    multi-homogenous components, by multiplying each variable by
    suitable scalars, and using a standard Vandermonde-type argument.
    Furthermore, in characteristic $0$, one can multilinearize any
    identity to an equivalent multilinear identity. Thus, in
    characteristic $0$ over a field, any T-ideal is generated by its
    multilinear part.

    Because of this, one considers the spaces
    \begin{equation}\label{Pndef}
    P_n=\{\sum_{\sigma\in S_n}\alpha_\sigma x_{\sigma(1)}x_{\sigma(2)}
    \cdots x_{\sigma(n)}\suchthat\alpha_\sigma\in C\}\end{equation} of
    multilinear polynomials in the variables $x_1,\dots,x_n$. This
    space has the structure of an $S_n$-module by defining:
    \[
        \tau\cdot x_{\sigma(1)}x_{\sigma(2)}\cdots x_{\sigma(n)}=x_{\tau\sigma(1)}x_{\tau\sigma(2)}\cdots x_{\tau\sigma(n)}.
    \]
    With the above definition, $C[S_n]\cong P_n$ as $S_n$-modules, with an isomorphism given by:
    $\sigma\mapsto x_{\sigma(1)}x_{\sigma(2)}\cdots x_{\sigma(n)}$.

    The multilinear part of degree $n$ of a T-ideal $\Gamma$ is given
    by $\Gamma\cap P_n$, which is an $S_n$-submodule of $P_n$. The
    quotient $P_n/{\Gamma\cap P_n}$ is called the \emph{$n$-th
    co-module} of $\Gamma$, and (in case $C=\F$ is a field) $c_n =
    \dim P_n/{\Gamma \cap P_n}$ is the \emph{$n$-th co-dimension}.

    The \emph{Specht problem} asks whether T-ideals are always \emph{finitely based}, namely generated as a T-ideal by some finite set. The Specht problem has been answered negatively for the analogous cases of groups and Lie algebras, 
    which made the following result by Kemer \cite[theorem~2.4]{kemer_monograph} quite surprising:

    \begin{thm}[Specht Property for algebras over fields] \label{thm:Specht_in_char_0}
    Let $A$ be an (associative) algebra over a field $C = \F$ of characteristic zero. Then the T-ideal $\T{A}$ is finitely based.
    \end{thm}

    This positive answer to the Specht problem in characteristic zero does not extend well to other
    characteristics, and has in fact been disproved for all non-zero characteristics. Additionally,
    there is no known method of actually finding the finite basis of the identities of a given algebra,
    and in fact, there are only a few natural cases where a complete basis of identities is known; even a basis
    for the identities of the matrix algebra of degree $3$ is unknown.

    Kemer proved his theorem via a series of reductions, first to the case of
    the T-ideal of identities of an affine algebra, and then it was shown that any T-ideal of identities of an affine algebra
    is also the T-ideal of identities of a finite-dimensional algebra.

    One concept of vital importance in the proof of \Tref{thm:Specht_in_char_0} is the \emph{Grassmann algebra}.
    The Grassmann algebra $G$ over a field $\F$ where $\chr{\F}\neq 2$ is the algebra generated by a countable set of generators
    $e_1,e_2,\dots$ under the relations:
    \begin{equation} \label{eq:comm_rel_for_grass_algebra}
        e_i e_j=-e_j e_i.
    \end{equation}

    \begin{rem} \label{rem:grassmann_identity}
        PI-theory in characteristic zero has quite a lot of information on $G$. For instance, it is known
        that when $\F$ is infinite, $\T{G}$ is generated by the single identity $[x,[y,z]]=0$ (this identity is known as the \emph{Grassmann
        identity}). Also, it is known that the co-dimension sequence of $G$ is exactly $c_n=2^{n-1}$. This
        result is obtained by first applying a combinatoric argument showing that the identity $[x,[y,z]]=0$
        has enough consequences to reduce the co-dimension to be $c_n\leq2^{n-1}$, and then using the
        representation theory of $S_n$ to show that it is bounded from below by the same amount.
    \end{rem}

    The structure of $G$ is related to the notion of \emph{superalgebras}: an algebra $A=A_0\oplus A_1$ satisfying
    $A_0 A_0\subseteq A_0$, $A_1 A_1\subseteq A_0$, $A_0 A_1\subseteq A_1$ and $A_1 A_0\subseteq A_1$ is called a
    \emph{superalgebra}. The subalgebra $A_0$ is called its \emph{even} part, and the $A_0$-module $A_1$ is called
    its \emph{odd} part. Additionally, the splitting $A=A_0\oplus A_1$ is referred to as the grading (or $\Z/2\Z$-grading) of $A$.
    Note that when we refer to a superalgebra $A$, we are actually referring to a specific grading $A=A_0\oplus A_1$,
    because in general there are many possible such gradings. An element $x$ in $A_0$ or $A_1$ is called \emph{homogenous}, and we
    let $|x|=0$ if $x\in A_0$ and $|x|=1$ if $x\in A_1$.

    The structure of $G$ now becomes trivial with respect to the following grading: we give $G$ the structure of
    a superalgebra by setting $G=G_0\oplus G_1$, where $G_0$ is the space spanned by all words of even length in
    the generators $e_1,e_2,\dots$ and $G_1$ is the space spanned by words of odd length.

    In general, if $A=A_0\oplus A_1$ is a superalgebra, then for all $x,y\in A$, let $x=x_0+x_1$, $y=y_0+y_1$, where
    $x_0,y_0\in A_0$, $x_1,y_1\in A_1$, be their decomposition into even and odd parts. Then define the
    \emph{supercommutator} of $x$ and $y$ by:
    \begin{equation*}
        \{x,y\}=[x_0,y_0]+[x_1,y_0]+[x_0,y_1]+(x_1 y_1+y_1 x_1),
    \end{equation*}
    where $[a,b]$ is the ordinary commutator. That is, when $x$ and $y$ are homogenous:
    \[
        \{x,y\}=xy-(-1)^{|x|\cdot |y|}yx.
    \]
    If the supercommutator of $x$ and $y$ is zero for all $x,y\in A$, then we
    say that $A$ is \emph{supercommutative}. Then with respect to the grading defined above, $G$ becomes supercommutative.

    One defines the \emph{free supercommutative algebra} $\freeG$ over $C$ as the superalgebra generated by countably many even generators $y_1,y_2,y_3,\dots$ and countably many odd generators $z_1,z_2,z_3,\dots$ whose only relations are $\{x_1,x_2\}=0$ for every $x_1,x_2 \in \freeG$. Note that $\freeG \cong  C[y_1,\dots] \otimes_C G$ as superalgebras, with the isomorphism given by $y_i \mapsto y_i \otimes 1$ and $z_i \mapsto 1\otimes e_i$. In particular when $C$ is an infinite field, $\T{S} = \T{G}$.

    One can build a theory of super linear algebra, with supertraces denoted by $\str$, superdeterminants (also known
    as Berezians) etc. (see \cite{notes_on_supertheory_bernstein,kantor_trish_det_in_supercase}). We merely note that the basic axiom of traces, $\tr{[a,b]}=0$, becomes, in the case of the supertrace, $\str{\{a,b\}}=0$ where $\{a,b\}$ is the
    supercommutator of $a$ and $b$. So, for example,

    \begin{dfn}
        If $A$ is any algebra with trace $\tr$, then the algebra $A\otimes G$ inherits the grading of $G$,
        and the function $\str(a\otimes w)=\tr(a)\otimes w$ becomes a supertrace. We will refer to this as the
        \emph{supertrace associated with $A\otimes G$}.
    \end{dfn}

    \begin{rem}
        This is a supertrace because of the easily verified fact that
        \[
            \{a\otimes w,b\otimes u\}=[a,b]\otimes wu,
        \]
        for all $a,b\in A$ and $u,w\in G$.
    \end{rem}

    In other words, tensoring by $G$ turns algebras into superalgebras, commutators into supercommutators, and traces into supertraces. The role of $G$ and superalgebras in general in PI-theory is best illustrated by the following
    deep theorem of Kemer, which reduces the study of arbitrary PI-algebras in characteristic $0$ to the study
    of finite-dimensional PI-superalgebras.

    \begin{thm}[Kemer's Superrepresentability Theorem] \label{thm:kemer_superrepresentability}
        For any algebra $A$ over a field of characteristic $0$, there is some finite-dimensional superalgbra $B$
        such that $\T{A}=\T{G[B]}$, where $G[B]=(G_0\otimes B_0)\oplus(G_1\otimes B_1)$ is the
        \emph{Grassmann hull} of $B$.
    \end{thm}

    The main problem with $G$ is that it cannot be easily generalized to arbitrary
    characteristics. In particular, in characteristic $2$ the relation \Eqref{eq:comm_rel_for_grass_algebra} implies that
    the algebra is commutative. For this reason, \cite{ces} came up with the following algebra, which was the basis for Belov's counterexample to the Specht problem in characteristic $2$ (see \cite[p.~204]{comp_aspects} for details):

    \begin{dfn} \label{def:G_+}
        Define the extended Grassmann algebra $G^+$ over a field $\F$ of characteristic $2$ as the algebra generated by
        elements $e_1,e_2,\dots$ and elements $\e_1,\e_2,\dots$ such that the $\e_i$ are central, and such that
        the following relation is satisfied:
        \begin{equation*}
            [e_i,e_j]=\e_i \e_j e_i e_j,
        \end{equation*}
        in addition to the relation:
        \begin{equation*}
            \e_i^2=0.
        \end{equation*}
        So, in fact, $G^+$ is an algebra over the local algebra $\F[\e_1,\e_2,\dots]$.
    \end{dfn}

    This algebra was used to produce counterexamples in characteristic $2$, such as constructing a T-ideal that is not finitely based (see for example \cite[p.~210, example~7.22]{comp_aspects}),
    as well as to investigate the T-space structure of the relatively free algebra generated by the Grassmann identity \cite{T_space_rels_of_Lie_nilp,T_space_multiplicative_structure_of_Grassmann,T_space_monomiality_equalization}.

    \begin{rem}
        The reason that this algebra is referred to as the extended Grassmann algebra is first of all that
        it is defined by relations similar to those that define the Grassmann algebra, and that its ideal of
        identities $\T{G^+}$ is generated by the same identity as the Grassmann algebra, $[x,[y,z]]=0$ (see
        \Rref{rem:grassmann_identity}).
    \end{rem}

    The main disadvantage of $G^+$ is that it is only non-degenerate in characteristic $2$, and superficially looks very
    different from the ordinary Grassmann algebra $G$. Therefore, our aim in this work is to present and study a version
    of the Grassmann algebra that is well-behaved over arbitrary commutative rings, which we denote as $\G$. We show that $\G$ possesses properties similar to the ordinary Grassmann algebra $G$,
    and generalize various theorems regarding $G$ over fields of characteristics $p\neq 2$ to theorems regarding $\G$ over rings of any characteristic.

The similarity to $G$ is demonstrated by the following two results:

\INTROTHM{\Tref{thm:all_idents_of_G_are_consequences_of_grassmann}}{
        Let $\G$ be the generalized Grassmann algebra defined over $C$. Then $\T{\G}$ is generated as a T-ideal by the Grassmann
        identity, $[x,[y,z]]=0$.}

\INTROTHM{\Tref{thm:G_and_extended_G_equivalence}}
{Suppose that $2$ is invertible in $C$. Let $A$ be some $C$-algebra. Then
        $\T{A\otimes_C S}=\T{A\otimes_C\G}$. In particular, $\T{\M[n](S)}=\T{\M[n](\G)}$.}

    And as a corollary, we have:

    \begin{cor}
        Suppose that $2$ is invertible in $C$. Then the ideal of identities of the free supercommutative algebra, $\T{\freeG}$,
        is generated as a T-ideal by the Grassmann identity.
    \end{cor}

    Next, we present a generalization of the notion of signs of permutations that is associated with $\G$ in much the same way ordinary signs are
    associated with the ordinary Grassmann algebra $G$. We refer to this generalization as the generalized sign representation, and show
    that the generalized sign representation is actually the whole co-module of $\G$, over any ring: The $S_n$-module of generalized signs $C[\e]_n$ over a ring $C$ is the $n$-th co-module of $\G$ (\Tref{thm:comodule_of_G}). Furthermore, we compute and show that the $n$-th co-module of $\G$ over a ring $C$ is a free $C$-module by the induced action of $C$, of rank $2^{n-1}$ (\Tref{thm:codim_of_G}). This generalizes the well known result that the co-dimension sequence of $G$ (in characteristic not $2$) is $c_n(G) = 2^{n-1}$.

    We continue to define a notion of generalized superalgebras, generalized Grassmann hulls and generalized supertraces (to which we refer
    as \super{}algebras and \super{}traces for brevity). The free \super{}algebra  $\efreeG$ is defined in \Exref{exm:efreeG}.  For the reader's convenience, let us collect here the notation used for the four objects studied and compared in this paper:
    \smallskip
    \begin{center}
        \begin{tabular}{c | c c}
        {} & superalgbera & \super{}algebra \\
        \hline
        Grassmann & $G$ & $\G$ \\
        free commutative & $\freeG$ & $\efreeG$
        \end{tabular}
    \end{center}
    \smallskip

    It is shown that when $2$ is invertible, these notions coincide with the notions of ordinary supertheory:

\INTROTHM{\Tref{thm:supertrace_and_extended_supertrace_equivalence}}{
\ThmTraceEquiv.
}

    The next question is what properties do supertraces
    (and more generally, \super{}traces) satisfy. Thus we turn our attention to the question of ungraded identities satisfied
    by supertraces. We find:

\INTROTHM{\Tref{thm:basis_of_identities_of_free_str}}
{The multilinear part of the ideal of identities of the free \super{}algebra with \super{}trace (over any ring) is generated
        by:
        \begin{eqnarray*}
            \esTr(\esTr(x)y)& = & \esTr(x)\esTr(y),\\
            \esTr(x\esTr(y))& = & \esTr(x)\esTr(y),\\
            {}[x,\esTr{[y,z]}] & = & 0,\\
            {}[\esTr(x),[\esTr(y),z]] & = & 0.
        \end{eqnarray*}
}

        \def\separatepaper{}

    \section{The Generalized Grassmann Algebra}

    The standard Grassmann algebra $G$ is well behaved in characteristic not $2$, while the generalized Grassmann algebra $G^+$ is defined in characteristic $2$. Our first objective is to combine the two objects into an algebra defined over an arbitrary (commutative) ring, in a way which is amenable to reductions and inverse limits.

    Starting from the relations $[e_i,e_j]=\e_i \e_j e_i e_j$ of
    \Dref{def:G_+}, we immediately obtain $-\e_i \e_j e_i
    e_j=-[e_i,e_j]=[e_j,e_i]=\e_i \e_j e_j e_i=\e_i
    \e_j(1-\e_i\e_j)e_i e_j$, which will be satisfied by requiring
    $-\e_i \e_j=\e_i \e_j(1-\e_i\e_j)$, or equivalently,
    $$\e_i^2\e_j^2=2\e_i\e_j.$$
    This observation motivates the following definition.
    \begin{dfn} \label{def:C_epsilon}
    We denote by $C[\e]$ the commutative ring $C[\e] = C[\theta,\e_1,\e_2,\dots]$, subject to the relations
      \[
        \e_i^2 = \theta\e_i.
      \]
        and
      \[
            \theta^2 = 2.
      \]
    \end{dfn}

    \begin{dfn} \label{def:generalized_extended_grassmann}
    	The \emph{generalized} Grassmann algebra $\G$ over $C$ is the unital algebra generated by
    	elements $e_1,e_2,\dots$ over the central subring $C[\e] = C[\theta,\e_1,\e_2,\dots]$ defined above, subject to the relations
    	\begin{equation} \label{comm_rel_for_ext_grass_algebra}
    		[e_i,e_j]=\e_i\e_j e_i e_j
    	\end{equation}
    	for every $i,j$ (in particular $\theta\e_i e_i^2 = \e_i^2 e_i^2 = 0$).
    \end{dfn}

    The following version of \eq{comm_rel_for_ext_grass_algebra} will be frequently used:
    \begin{equation} \label{rem:sign_of_transposition}
    	e_j e_i=(1-\e_i\e_j)e_i e_j.
    \end{equation}

    \begin{rem}
    	The elements $e_j^2$ are central, as
    	\[
    		e_j^2e_i =(1-\e_i \e_j)^2e_i e_j^2=(1-2\e_i\e_j+\e_i^2\e_j^2)e_i e_j^2=e_i e_j^2.
    	\]
    \end{rem}

    Modulo $\theta$ we recover the extended Grassmann algebra. More precisely, the quotient $\G/\theta\G$ is the extended  Grassmann algebra $G^+$ over $C/2C$.

    The terminology attached to $\G$ is justified by the following theorem.

    \begin{thm} \label{thm:all_idents_of_G_are_consequences_of_grassmann}
        Let $\G$ be the generalized Grassmann algebra defined over $C$. Then $\T{\G}$ is generated as a T-ideal by the Grassmann
        identity, $[x,[y,z]]=0$.
    \end{thm}

    We first show that $[x,[y,z]]=0$ holds in $\G$, and then that all other identities of $\G$ are consequences of it.

    \begin{lem} \label{lem:grass_ident_for_generators}
        Let $e_1,e_2,\dots\in\G$ be the generators as in \Dref{def:generalized_extended_grassmann}. Then,
        \begin{enumerate}
            \item $[e_i,[e_j,e_k]]=0$ for all $i$,$j$ and $k$.
            \item $[e_i,e_j][e_m,e_k]+[e_j,e_k][e_i,e_m]=0$ for all $i$,$j$,$k$ and $m$.
        \end{enumerate}
    \end{lem}

    \begin{proof}
       We have:
        \begin{eqnarray*}
            {}[e_i,[e_j,e_k]] & =& [e_i,\e_j\e_k e_j e_k] \\
            & = & \e_j\e_k[e_i,e_j e_k] \\
            & = & \e_j\e_k([e_i,e_j]e_k+e_j[e_i,e_k])\\
            & = & \e_j\e_k(\e_i\e_j e_i e_j e_k+\e_i\e_k e_j e_i e_k)\\
            & = &(\e_i\e_j^2\e_k+\e_i\e_j\e_k^2-\e_i^2\e_j^2\e_k^2)e_i e_j e_k\\
            & = & (\theta+\theta-\theta^3)\e_i\e_j\e_k e_i e_j e_k\\
            & =& (2\theta-2\theta)\e_i\e_j\e_k e_i e_j e_k=0.
        \end{eqnarray*}
        Similarly,
        \begin{eqnarray*}
            {}[e_i,e_j][e_m,e_k] & + & [e_j,e_k][e_i,e_m]=\e_i\e_j\e_m\e_k(e_i e_j e_m e_k+e_j e_k e_i e_m)\\
            & = & \e_i\e_j\e_m\e_k(1+(1-\e_i\e_k)(1-\e_j\e_i)(1-\e_k\e_m))e_i e_j e_m e_k\\
            & = & \e_i\e_j\e_m\e_k(1+(1-\theta^2)(1-\theta^2)(1-\theta^2))e_i e_j e_m e_k\\
            & = & \e_i\e_j\e_m\e_k(1-1)e_i e_j e_m e_k=0.
        \end{eqnarray*}
    \end{proof}

    More generally:
    \begin{lem} \label{lem:grass_consequ_for_one_word}
        We have $[e_i,e_j][u,e_k]+[e_j,e_k][e_i,u]=0$ for every element $u \in \G$.
    \end{lem}

    \begin{proof}
        It suffices to check the claim for monomials. Let $u=e_{\ell_1}\cdots e_{\ell_n}$. Then, we have:
        \begin{multline*}
            [e_i,e_j][u,e_k]+[e_j,e_k][e_i,u]=\\
            \displaystyle\sum_m e_{\ell_1}\cdots e_{\ell_{m-1}}([e_i,e_j][e_{\ell_m},e_k]+[e_j,e_k][e_i,e_{\ell_m}])e_{\ell_{m+1}}\cdots e_{\ell_{n}}=0,
        \end{multline*}
     by \Lref{lem:grass_ident_for_generators}.
    \end{proof}

    \begin{lem} \label{lem:G_satisfies_grass}
        We have that $\G$ satisfies the Grassmann identity.
    \end{lem}

    \begin{proof}
        We wish to show that all commutators are central. Thus, it suffices to show that they commute with the $e_i$-s.
        So, we must show that $[e_i,[w_1,w_2]]=0$ where $w_1$ and $w_2$ are some words in the generators. If the lengths
        of both $w_1$ and $w_2$ are $1$, then we are done by the previous lemma. Otherwise, assume without loss of generality
        that $w_1=e_j u$, and assume via induction that we already have: $[e_i,[x,y]]=0$ for all $i$ and for all words $x$, $y$
        such that $x$ is not longer than $u$, and $y$ is not longer than $w_2$. Then
        \begin{eqnarray*}
            {}[e_i,[e_j u,w_2]] & = & [e_i,e_j[u,w_2]]+[e_i,[e_j,w_2]u] \\
            & = &[e_i,e_j][u,w_2]+e_j[e_i,[u,w_2]]+[e_i,[e_j,w_2]]u+[e_j,w_2][e_i,u]\\
            & = & [e_i,e_j][u,w_2]+[e_j,w_2][e_i,u].
        \end{eqnarray*}

        We need to prove that this is zero. We will do so by induction. If $w_2=e_k v$, and if we assume that the expression is zero for all
        shorter words, then
        \begin{gather*}
            [e_i,e_j][u,w_2]+[e_j,w_2][e_i,u]=[e_i,e_j][u,e_k v]+[e_j,e_k v][e_i,u]=\\
            [e_i,e_j][u,e_k]v+[e_i,e_j]e_k[u,v]+e_k[e_j,v][e_i,u]+[e_j,e_k]v[e_i,u]=\\
            e_k([e_i,e_j][u,v]+[e_j,v][e_i,u])+v([e_i,e_j][u,e_k]+[e_j,e_k][e_i,u])
        \end{gather*}
        since $e_k$, $v$ commute with the commutators (by the outer induction hypothesis). We are thus left with proving that
        $[e_i,e_j][u,e_k]+[e_j,e_k][e_i,u]=0$, which also serves as the basis of the (inner) induction. But this is exactly
        what we have already proven in \Lref{lem:grass_consequ_for_one_word}.
    \end{proof}

    We are now left with proving the other direction of \Tref{thm:all_idents_of_G_are_consequences_of_grassmann}.

    \begin{rem}[{\cite[Lemmas~3.43 and 3.44]{comp_aspects}}] \label{lem:G_satisfies_grass_consequ}
        The identities
        \begin{eqnarray}
            {}[x,u][v,z]+[x,v][u,z] & = & 0, \label{eq:Id1}\\
            {}[x,y][y,z] & =& 0 \label{Id2} \nonumber
        \end{eqnarray}
        are consequences of the Grassmann identity.
    \end{rem}

    \begin{lem} \label{lem:all_idents_of_G_are_consequences_of_grassmann}
        All identities of $\G$ are consequences of the Grassmann identity.
    \end{lem}

    \begin{proof}
        We would first like to reduce to the multi-homogenous case. So, note that $G/\ideal{\e_i\suchthat i\in X}$, for all finite
        $X\subseteq\N$, is isomorphic to $C[\lambda_i\suchthat i\in X]\otimes_C\G$, where $C[\lambda_i\suchthat i\in X]$ is a commutative
        polynomial algebra in $|X|$ variables. Thus, if $f(x_1,\dots,x_n)$ is an identity, then $f(\lambda_1\otimes x_1,\lambda_2\otimes x_2,\dots,\lambda_n\otimes x_n)$ is
        also an idnetity. If we let $f_{d_1,\dots,d_n}(\lambda_1\otimes x_1,\dots,\lambda_n\otimes x_n)$ be the component of $f(\lambda_1\otimes x_1,\dots,
        \lambda_n\otimes x_n)$ of degree $d_i$ in $\lambda_i$, we see that $f_{d_1,\dots,d_n}(\lambda_1\otimes x_1,\dots,\lambda_n\otimes x_n)=
        \lambda_1^{d_1}\cdots\lambda_n^{d_n}\otimes f_{d_1,\dots,d_n}(x_1,\dots,x_n)$ are the multihomogenous components of $f$, and must be
        equal to zero separately. Thus, we can assume that $f$ is multi-homogenous.

        So, let $f$ be a multi-homogenous identity of $\G$. We need to prove that it is a consequence of the Grassmann identity. Since
        commutators are central, $f$ can be rewritten as a sum of terms of the form
        \[
            a x_{k_1}\cdots x_{k_m}[x_{k_{m+1}},x_{k_{m+2}}][x_{k_{m+3}},x_{k_{m+4}}]\cdots[x_{k_{n-1}},x_{k_n}],
        \]
        where $k_1\leq\dots\leq k_m$. Using \Eqref{eq:Id1}, we may assume that $k_{m+1} < \dots <x_{k_n}$.

        Substitution of $1$ for all of $x_1,\dots,x_n$ sends $f$ to  the
        coefficient of the term $x_1\cdots x_n$, and since $f$ is an
        identity, this coefficient is zero. For every pair of variables
        $x_i,x_j$, substitute $1$ for the other variables and $e_1,e_2$
        for $x_i,x_j$; the only nonzero term is the one in which exactly
        these two variables are in the commutator, which again proves that
        the coefficient of this term is zero. Repeating this argument for
        all subsets of four variables, then six, and so on, we see that
        $f$ is zero modulo the Grassmann identity.
    \end{proof}

    \ifx\separatepaper\undefined
        \begin{rem}\label{rem:mode2}
            The same proof works for $\G/\ideal{e_i^2}$, so in fact we have $\T{\G}=\T{\G/\ideal{e_i^2}}$.
        \end{rem}
    \fi

    \subsection{The Ring \texorpdfstring{$C[\e]$}{C[e]} and the Connection to the Grassmann Algebra} \label{subsec:idemps}

    Our next goal is to show that when $2$ is invertible, $C[\e]$ has enough idempotents to break $\G$ into a sum of supercommutative
    pieces. The basic observation is that the expressions $\frac{1}{2}\theta\e_i$ (if defined) are idempotents.

    \begin{dfn} \label{def:restricted_G_+}
    For any subset $X\subseteq\N$, let
    $\G_X=C\sg{e_j,\e_j,\theta\suchthat j\in X}\subset\G$ be the
    subalgebra generated by all generators $\e_j$ and $e_j$ whose
    indices are in $X$.
    \end{dfn}

    \begin{dfn} \label{def:idemps}
    	Assume that $2$ is invertible in $C$, and let $X\subseteq\N$ be a
    	finite subset. For any association $s \co X\rightarrow\set{\pm 1}$
    	of signs to the indices in $X$, define
    	\begin{equation*}
    		\Lambda_s=\displaystyle\prod_{s(a) = -1}\frac{1}{2}\theta\e_a\prod_{s(b) = +1}(1-\frac{1}{2}\theta\e_b).
    	\end{equation*}
    \end{dfn}

    \begin{prop} \label{prop:idempotents}
        Assume that $2$ is invertible in $C$.
        Let $X\subseteq\N$ be a finite subset.
        \begin{enumerate}
            \item \label{item_idemp_complete} The elements $\Lambda_s\in C[\e]$, for $s \co X \ra \set{\pm 1}$, form a complete system of
                idempotents of $C[\e]$.
            \item \label{item_idemp_behavior} For every $s \co X \ra \set{\pm 1}$, the algebra $\Lambda_s\G_X$ is a free supercommutative
                algebra, with even generators $\theta$ and $\Lambda_s e_b$ for $s(b)=+1$, and odd generators $\Lambda_s e_a$ for $s(a)=-1$.
        \end{enumerate}
    \end{prop}

    \begin{proof}
        The defining relations imply that the elements
        $\frac{1}{2}\theta\e_i$ are idempotents, from which it follows
        that every $\Lambda_s$ is an idempotent. Furthermore
        \[
                \displaystyle\sum_{s \co X \ra \set{\pm 1}}\Lambda_s=\prod_{i \in X}\((\frac{1}{2}\e_i\theta)+(1-\frac{1}{2}\e_i\theta)
                    \)= 1.
        \]

      For \ref{item_idemp_behavior}, let $a,a',b,b' \in X$ be such that $s(a)=s(a')=-1$, $s(b)=s(b')=+1$. We have:
      \begin{eqnarray*}
          [e_{a}\Lambda_s,e_{a'}\Lambda_s] & = & \e_{a}\e_{a'} e_{a} e_{a'}\Lambda_s\\
          & = & \e_{a}\e_{a'} e_{a} e_{a'}\frac{1}{2}\theta\e_{a}\frac{1}{2}\theta\e_{a'}\Lambda_s\\
          & = & \frac{1}{4}\theta^2\e_{a}^2\e_{a'}^2 e_{a} e_{a'}\Lambda_s\\
          & = & 2e_{a} e_{a'}\Lambda_s,
      \end{eqnarray*}
        So $\Lambda_se_{a}$ and $e_{a}\Lambda_s$ anticommute. The proof
        that $\Lambda_s e_b$ are central is analogous. Freeness then
        easily follows.
    \end{proof}

    Multiplying by a suitable idempotent, we may thus declare finitely
    many of the $e_1,e_2,\dots$ even, and finitely many others, odd.
    With this new understanding, we can now prove a much stronger
    correspondence between $\G$ and $\freeG$:

    \begin{thm} \label{thm:G_and_extended_G_equivalence}
        Suppose that $2$ is invertible in $C$. Let $A$ be some $C$-algebra. Then
        $\T{A\otimes_C \freeG}=\T{A\otimes_C\G}$. In particular, $\T{\M[n](\freeG)}=\T{\M[n](\G)}$.
    \end{thm}

    \begin{proof}
        We first show that any identity of $A\otimes\G$ is an identity of $A\otimes \freeG$. Indeed, define
        a homomorphism of $C$-algebras, $\phi \co \freeG\rightarrow\G$, by
        $\phi(e_a)=\frac{1}{2}\theta\e_a e_a\in\G$ for odd generators $e_a$, and $\phi(e_b)=(1-\frac{1}{2}\theta\e_b)e_b\in\G$ for
        even generators $e_b$ (note that the $e_i$ on the left hand side of this equation are elements from $\freeG$, and on the right
        hand side from $\G$). This homomorphism is clearly injective. Since
        $\freeG$, $\G$ and the image of $\phi$ are all free $C$-modules, and the image of a base of $\freeG$ under $\phi$ can be
        completed to a base of $\G$ (by considering the base of words in $\freeG$, and the base
        of words multiplied by all idempotents associated to generators in the word, possibly times $\theta$), we
        see that the map $\text{1}_A\otimes\phi:A\otimes \freeG\rightarrow A\otimes\G$ is an injective homomorphism (indeed
        $\G/\phi(\freeG)$ is a free $C$-module, so $\Tor_1^C(A,\G/\phi(\freeG))=0$). Thus,
        $\T{A\otimes_C \freeG}\supseteq\T{A\otimes_C\G}$.

            In the other direction, let $f\in\T{A\otimes_C \freeG}$, and let $x_i\mapsto\hat x_i\in\G$ be a substitution of elements from $\G$ in the variables appearing in $f$.
            Let $X$ be the (finite) collection of all the indices $j$ of all $e_j$ or $\e_j$ appearing in
            some of the $\hat x_i$. Recall the definition of the subalgebra $\G_X=C\sg{e_j,\e_j,\theta\suchthat j\in X}\subset\G$.
            By \Pref{prop:idempotents}, the idempotents $\Lambda_s$, with $s \co X \ra \set{\pm 1}$, form a complete set of idempotents for
            $\G$ (and thus $\G_X$). Then it is sufficient to consider substitutions $x_i\mapsto
        \Lambda_s\hat x_i\in\Lambda_s\G_X$ for some fixed $s \co X \ra \set{\pm 1}$. But now, \Pref{prop:idempotents} shows that
        $\Lambda_s\G_X$ is a free supercommutative algebra,
        so we can fix a canonical embedding $\psi\co \Lambda_s\G_X\rightarrow \freeG$ of $\Lambda_s\G_X$ in $\freeG$. Again, we see that it maps the base of $\Lambda_s\G_X$ into a set that can be completed to a
        base of $\freeG$ (take the base generated by $\Lambda_s$ times words in $\G_X$, and the base of words in $\freeG$).
        Hence, the map
        \[
            \text{id}_A\otimes\psi:A\otimes\Lambda_s\G_X\rightarrow A\otimes \freeG
        \]
        is an injective homomorphism. Thus, $f$ is zero on substitutions from $\Lambda_s\G_X$ and is therefore zero
        on the substitution $x_i\mapsto\Lambda_s\hat x_i\in\Lambda_s\G_X$. This completes the proof.
    \end{proof}

    \begin{rem}
        Over a field $C=\F$, this would follow from the case $A=\F$, or $\T{\freeG}=\T{\G}$, since all $\F$-modules are flat.
    \end{rem}

    \begin{rem}
    	Over a finite field, $\T{G}$ strictly contains $\T{\freeG}$. For example over $C = \F_3$, the polynomial $x^9y^3-x^3y^9$ is an identity of $G$, which does
    	not follow from the Grassmann identity.
    	
    	Indeed, working modulo $3$, if $x=x_0+x_1$ is the decomposition of $x$ to homogenous parts, then: $x^3=x_0^3+x_1^3=x_0^3$. But, the even
    	part of $G$ is spanned by $1$ and words of positive even length, so writing $x_0 = \lambda+w$, where $\lambda\in\Z_3$, we have $x^3=x_0^3=\lambda^3=\lambda$.
    	Thus, the identity becomes $\lambda\mu(\lambda^2-\mu^2)$, which is an identity of $\Z_3$. A similar construction works over any finite field.
    \end{rem}

    As an immediate corollary, we now have a proof of the following theorem, proved by Regev and Krakowsky in characteristic $0$
    \cite{Grassman_id}, and by Giambruno and Koshlukov in characteristic $p\neq 2$ \cite{Grassman_id_in_positive_char}.

    \begin{cor}
        Suppose that $2$ is invertible in $C$. Then the ideal of identities of the free supercommutative algebra, $\T{\freeG}$,
        is generated as a T-ideal by the Grassmann identity.
    \end{cor}

    \begin{proof}
        According to \Tref{thm:G_and_extended_G_equivalence}, in this case $\T{\freeG}=\T{\G}$. But we have already seen that $\T{\G}$ is generated by
        the Grassmann identity (see \Tref{thm:all_idents_of_G_are_consequences_of_grassmann}).
    \end{proof}

    \subsection{Generalized Signs}

    Now that we have a clear understanding of the role taken by the $\e_i$-s, we can introduce some helpful notation.

    \begin{dfn}
        Define the map
        \[
            \exp:\text{span}_{\Z_2}\{\e_i \e_j\suchthat i,j\in\N\}\rightarrow C[\e]
        \]
        by
        \begin{enumerate}
            \item $\exp(0)=1$,
            \item $\exp(\e_i \e_j)=1-\e_i \e_j$,
            \item $\exp(a+b)=\exp(a)\exp(b)$.
        \end{enumerate}
        In addition, if $w\in\G$ is a word in the generators, $w=e_{i_1}\cdots e_{i_n}$, then define: $\e_w=\e_{i_1}+\dots+\e_{i_n}$.
        Clearly, for any two such words $w$ and $w'$, we have $\e_w\e_{w'} \in \text{span}_{\Z_2}\{\e_i \e_j\}$.
    \end{dfn}

    \begin{rem}\label{rem:exp2}
    The exponent, a-priori defined on $\text{span}_{\Z}\{\e_i \e_j\suchthat i,j\in\N\}$, is well defined over $\Z_2$
    because $\exp(2\e_i \e_j)=(1-\e_i \e_j)^2=1-2\e_i \e_j+\e_i^2 \e_j^2=1-2\e_i \e_j+2\e_i \e_j=1$. For the same reason, $\exp(a)^2 = \exp(2a) = 1$ for every $a$.
    \end{rem}

    The following computation generalizes \Rref{rem:sign_of_transposition}.
    \begin{prop} \label{prop:exponent_property}
    	For any two monomials $u,w\in\G$ in the generators $e_i$,
    	\[
    		uw=\exp(\e_u\e_w)wu.
    	\]
    \end{prop}

    \begin{proof}
        \Rref{rem:sign_of_transposition} proves the case $u=e_i$, $w=e_j$. Let us verify the claim for $u=e_i$, $w=e_{j_1}\cdots e_{l_m}$. Indeed, we see that
        \begin{equation*}
            \begin{split}
                uw=e_i e_{j_1}e_{j_2}\cdots e_{j_m}&=\exp(\e_i\e_{j_1})e_{j_1}e_i e_{j_2}\cdots e_{j_m}\\
                &=\exp(\e_i\e_{j_1})\exp(\e_i\e_{j_2})e_{j_1}e_{j_2}e_i\cdots e_{j_m}=\dots\\
                &=\exp(\e_i\e_{j_1})\exp(\e_i\e_{j_2})\cdots\exp(\e_i\e_{j_m})e_{j_1}e_{j_2}\cdots e_{j_m}e_i\\
                &=\exp(\e_i(\e_{j_1}+\dots+\e_{j_m}))wu\\
                &=\exp(\e_u\e_w)wu.
            \end{split}
        \end{equation*}

        Now, let $u=e_{i_1}\cdots e_{i_n}$, $w=e_{j_1}\cdots e_{j_m}$. Then:
        \begin{eqnarray*}
            uw & = & e_{i_1}\cdots e_{i_{n-1}}e_{i_n}w \\
            & = &   \exp(\e_{i_n}\e_w)e_{i_1}\cdots e_{i_{n-1}}we_{i_n}=\exp(\e_{i_n}\e_w)\exp(\e_{i_{n-1}}\e_w)e_{i_1}\cdots we_{i_{n-1}}e_{i_n} \\
            & = & \dots=\exp(\e_{i_n}\e_w)\cdots\exp(\e_{i_1}\e_w)we_{i_1}\cdots e_{i_n}=\exp(\e_u\e_w)wu.
        \end{eqnarray*}
    \end{proof}

    Let us introduce a further generalization of the exponent map,
    which we call a generalized sign. We use the natural action of the
    infinite symmetric group $S_\N$ on $C[\e]$ by $\phi_\s(\theta) =
    \theta$ and
    \[
    	\phi_\sigma(\e_i)=\e_{\sigma(i)}.
    \]

    \begin{dfn}
    		Let $w=(w_1,\dots,w_n)$ be an $n$-tuple of words in the generators
    		$e_i$. For $\sigma\in S_n$, a permutation on the set
    		$\{1,\dots,n\}$, we define the \emph{generalized sign} to be:
        \begin{equation*}
             \esgn{w}{\sigma}=\exp\left(\displaystyle\sum_{\stackrel{i<j}{\sigma(i)>\sigma(j)}}\e_{w_{\sigma(i)}}\e_{w_{\sigma(j)}}\right).
        \end{equation*}
    \end{dfn}

    \begin{prop} \label{prop:sign_properties}
    		Let $w=(w_1,\dots,w_n)$ be a $n$-tuple of words in the generators
    		$e_i$.
        \begin{enumerate}
            \item \label{sign_main_property} For every $\s \in S_n$,
                \begin{equation*}
                    w_{\sigma(1)}w_{\sigma(2)}\cdots w_{\sigma(n)}=\esgn{w}{\sigma}w_1 w_2\cdots w_n.
                \end{equation*}
            \item \label{general_sign_left_multiplication}
    						For every $\sigma,\tau\in S_n$,
                \begin{equation*}
                     \esgn{w}{\sigma\tau}=\esgn{w}{\sigma}\esgn{\sigma(w)}{\tau}
                \end{equation*}
     						where $\sigma(w)=(w_{\sigma(1)},\dots,w_{\sigma(n)})$.
            \item \label{sign_left_multiplication} In particular, when $w=(e_1,\dots,e_n)$,
                \begin{equation*}
                     \esgn{w}{\sigma\tau}=\esgn{w}{\sigma}\phi_\sigma(\esgn{w}{\tau}).
                \end{equation*}
        \end{enumerate}
    \end{prop}

    \begin{proof}
    		Write $\sigma=s_1\cdots s_m$ where $s_j=(k_j,k_j+1)$ are Coxeter
    		generators of $S_n$. We prove \ref{sign_main_property} by
    		induction on $m$. For $m=0$, the claim is trivial. Assume the
    		claim holds for $\pi=s_1\cdots s_{m-1}$. Then according to
    		\Pref{prop:exponent_property} and since $s_m$ transposes
    		$w_{\pi(k_m)}$ and $w_{\pi(k_m+1)}$, we have:
    		\begin{eqnarray*}
             w_{\sigma(1)}w_{\sigma(2)}\cdots w_{\sigma(n)}
            & = & w_{\pi s_m(1)}w_{\pi s_m(2)}\cdots w_{\pi s_m(n)}\\
            & = & w_{\pi(1)}w_{\pi(2)}\cdots w_{\pi(k_m-1)}w_{\pi(k_m+1)}w_{\pi(k_m)}w_{\pi(k_m+2)}\cdots w_{\pi(n)}\\
            & = & \exp\left(\e_{w_{\pi(k_m)}}\e_{w_{\pi(k_m+1)}}\right)w_{\pi(1)}w_{\pi(2)}\cdots
              w_{\pi(k_m-1)}w_{\pi(k_m)}w_{\pi(k_m+1)}\cdots w_{\pi(n)}\\
            & = & \exp\left(\e_{w_{\sigma(k_m)}}\e_{w_{\sigma(k_m+1)}}\right)w_{\pi(1)}w_{\pi(2)}\cdots w_{\pi(n)},\\
        		& = & \exp\left(\e_{w_{\sigma(k_m)}}\e_{w_{\sigma(k_m+1)}}\right)\esgn{w}{\pi}w_1 w_2 \cdots w_{n},
        \end{eqnarray*}
    		where the last equality follows from the induction hypothesis.
    		Acting by $s_m=(k_m,k_m+1)$ does not affect the order of any of the pairs
    		$i<j$, except for flipping the order
        of the pair $k_m,k_m+1$. Thus,
        \begin{eqnarray*}
             \exp\left(\e_{w_{\sigma(k_m)}}\e_{w_{\sigma(k_m+1)}}\right)\esgn{w}{\pi}
            & = & \exp\left(\e_{w_{\sigma(k_m)}}\e_{w_{\sigma(k_m+1)}}\right)\exp\left(\displaystyle\sum_{\stackrel{i<j}
                {\pi(i)>\pi(j)}} \e_{w_{\pi(i)}}\e_{w_{\pi(j)}}\right)\\
            & = & \exp\left(\e_{w_{\sigma(k_m)}}\e_{w_{\sigma(k_m+1)}}+\displaystyle\sum_{\stackrel{s_m(i)<s_m(j)}
                {\pi s_m(i)>\pi s_m(j)}} \e_{w_{\pi s_m(i)}}\e_{w_{\pi s_m(j)}}\right)\\
            & = & \exp\left(\e_{w_{\sigma(k_m)}}\e_{w_{\sigma(k_m+1)}}+\displaystyle\sum_{\stackrel{s_m(i)<s_m(j)}
                {\sigma(i)>\sigma(j)}} \e_{w_{\sigma(i)}}\e_{w_{\sigma(j)}}\right)\\
            & = & \exp\left(\displaystyle\sum_{\stackrel{i<j}{\sigma(i)>\sigma(j)}}\e_{w_{\sigma(i)}}\e_{w_{\sigma(j)}}\right)=\esgn{w}{\sigma},
        \end{eqnarray*}
        as claimed.

    		To prove~\ref{general_sign_left_multiplication}, we compute
        \begin{align*}
            \esgn{w}{\sigma} \esgn{\sigma(w)}{\tau} & =
            & \exp\left(\displaystyle\sum_{\stackrel{i<j}{\sigma(i)>\sigma(j)}}\e_{w_{\sigma(i)}}\e_{w_{\sigma(j)}}+
                 \sum_{\stackrel{i<j}{\tau(i)>\tau(j)}}\e_{\sigma(w)_{\tau(i)}}\e_{\sigma(w)_{\tau(j)}}\right) \\
           & = &\exp\left(\displaystyle\sum_{\stackrel{i<j}{\sigma(i)>\sigma(j)}}\e_{w_{\sigma(i)}}\e_{w_{\sigma(j)}}+
                 \sum_{\stackrel{i<j}{\tau(i)>\tau(j)}}\e_{w_{\sigma\tau(i)}}\e_{w_{\sigma\tau(j)}}\right)\\
            & = &\exp\left(\displaystyle\sum_{\stackrel{\tau(i)<\tau(j)}{\sigma\tau(i)>\sigma\tau(j)}}\e_{w_{\sigma\tau(i)}}\e_{w_{\sigma\tau(j)}}+
                 \sum_{\stackrel{i<j}{\tau(i)>\tau(j)}}\e_{w_{\sigma\tau(i)}}\e_{w_{\sigma\tau(j)}}\right).
        \end{align*}
        But since each pair $i<j$ whose order is inverted by $\sigma\tau$ is inverted by $\sigma$ or by $\tau$, we have that
        \begin{multline*}
                 \exp\left(\displaystyle\sum_{\stackrel{\tau(i)<\tau(j)}{\sigma\tau(i)>\sigma\tau(j)}}\e_{w_{\sigma\tau(i)}}\e_{w_{\sigma\tau(j)}}+
                 \sum_{\stackrel{i<j}{\tau(i)>\tau(j)}}\e_{w_{\sigma\tau(i)}}\e_{w_{\sigma\tau(j)}}\right)=\\
            \exp\left(\displaystyle\sum_{\stackrel{i<j}{\sigma\tau(i)>\sigma\tau(j)}}\e_{w_{\sigma\tau(i)}}\e_{w_{\sigma\tau(j)}}\right)=
            \esgn{w}{\sigma\tau}.
        \end{multline*}
    \end{proof}

    In order to see that the generalized sign $\esgn{w}{\cdot}$ is correct generalization of the notion of signs, note that in $G$,
    we have $e_{\sigma(1)}\cdots e_{\sigma(n)} =\sgn{\sigma}e_1\cdots e_n$. Furthermore, the idempotent corresponding to the constant
    function $s(i) = -1$ ($i=1,\dots,n$) satisfies
    \[
    	\esgn{(e_1,\dots,e_n)}{\sigma}\Lambda_s=\sgn{\sigma}\Lambda_s,
    \]
    since the $e_i$ anticommute in the presence of $\Lambda_{s}$.

    \subsection{The Co-module Sequence of \texorpdfstring{$\G$}{G}}

    We now turn our attention to the co-modules and
    co-dimensions of $\G$. We begin by defining an
    $S_n$-representation analogous to the usual sign representation:

    \begin{dfn}
    Fix $w=(e_1,\dots,e_n)$. We consider the natural action of $S_n$
    on $C[\e]$ twisted by signs: For each $\sigma\in S_n$ and
    $\lambda\in C[\e]$,
        \begin{equation}\label{X2}
            \sigma(\lambda)=\esgn{w}{\sigma}\phi_\sigma(\lambda).
        \end{equation}
    Also let $C[\e]_n$ denote the $S_n$-submodule of $C[\e]$ generated
    as a module by $1\in C[\e]$.
    \end{dfn}

    \begin{rem}
        According to \Pref{prop:sign_properties}.\ref{sign_left_multiplication}, this indeed gives $C[\e]$ an $S_n$-module
        structure, as
        \begin{eqnarray*}
            (\s\tau)(\lambda) & = & \esgn{w}{\sigma\tau}\phi_{\sigma\tau}(\lambda)\\
            & = & \esgn{w}{\sigma}\phi_\s(\esgn{w}{\tau})\phi_\s(\phi_\tau(\lambda))\\
            & = & \esgn{w}{\sigma}\phi_\s(\esgn{w}{\tau}\phi_\tau(\lambda))\\
            & = & \sigma(\esgn{w}{\tau}\phi_{\tau}(\lambda)) = \sigma(\tau(\lambda)).
        \end{eqnarray*}
    \end{rem}

    \begin{exm} \label{exm:extended_signs_of_size_3}
        Consider the $S_3$-module $C[\e]_3$. By definition $C[\e]_3$ is spanned as a $C$-module by the elements $\sigma(1)=\esgn{w}{\sigma}$:
        \begin{eqnarray*}
            \esgn{w}{1} & = & 1,\\
            \esgn{w}{(1\,2)} & = & \exp(\e_1\e_2)=1-\e_1\e_2,\\
            \esgn{w}{(2\,3)} & = & \exp(\e_2\e_3)=1-\e_2\e_3,\\
            \esgn{w}{(1\,3)} & = & \exp(\e_1\e_2+\e_2\e_3+\e_1\e_3)=(1-\e_1\e_2)(1-\e_2\e_3)(1-\e_1\e_3)\\
                & = & 1-\e_1\e_2-\e_2\e_3-\e_1\e_3+\theta\e_1\e_2\e_3,\\
            \esgn{w}{(1\,2\,3)} & = & \exp(\e_1(\e_2+\e_3))=(1-\e_1\e_2)(1-\e_1\e_3)\\
                & = & 1-\e_1\e_2-\e_1\e_3+\theta\e_1\e_2\e_3,\\
                \esgn{w}{(1\,3\,2)} & = & \exp(\e_3(\e_1+\e_2))=(1-\e_1\e_3)(1-\e_2\e_3)\\
                & = & 1-\e_1\e_3-\e_2\e_3+\theta\e_1\e_2\e_3.
        \end{eqnarray*}
        Therefore, $C[\e]_3$ is a free $C$-module of rank $4$, spanned by $1$, $\e_1\e_2$, $\e_2\e_3$ and
        $\e_1\e_3-\theta\e_1\e_2\e_3$.
    \end{exm}

    We can now state the main result of this section.

    \begin{thm} \label{thm:comodule_of_G}
        The $n$-th co-module of $\G$ is isomorphic, as an $S_n$-module, to $C[\e]_n$.
    \end{thm}

    To prove the theorem, we will first establish that a multilinear polynomial that vanishes on $e_1,\dots,e_n$ vanishes on any
    other substitution. Since $S_n$ acts on the space $P_n$ defined in \eq{Pndef} by reordering variables, and since reordering
    variables multiplies by the generalized sign, \Tref{thm:comodule_of_G} follows (as will be explained below).

    We observe that $\G$ has plenty of endomorphisms.

    \begin{lem} \label{lem:generization_of_G}
        For any $n$-tuple of words $w=(w_1,\dots,w_n)$ in the generators $e_i$, there is a morphism $\eta_w:\G\rightarrow\G$ such that
        for all $1\leq i\leq n$:
        \begin{equation*}
            \eta_w(e_i)=w_i.
        \end{equation*}
    \end{lem}

    \begin{proof}
    	First we show that for every $\l$ and for every word $w$ of length
    	$1$ or $2$, there is a homomorphism of $C$-algebras $\G \ra \G$
    	such that $e_i \mapsto e_i$ and $e_\l \mapsto w$.
    	
    	Indeed, when $w = e_j$, define the map on $C[\e]$ by $\theta
    	\mapsto \theta$, $\e_i \mapsto \e_i$ for every $i \neq \l$, and
    	$\e_l \mapsto \e_j$. This is easily seen to be well defined.
    	
    	Likewise when $w = e_j e_k$, define the map $\eta_{j,k,\l}$ by $\theta\mapsto\theta$, $\e_i\mapsto\e_i$, $e_i\mapsto e_i$ for
    	$i\neq\l$, and $\e_\l\mapsto\e_j+\e_k-\theta\e_j\e_k$, $e_\l\mapsto e_je_k$. In order to show that this homomorphism is well
    	defined, it suffices to check the relations $\e_\l^2=\theta \e_l$ and $[e_\l,e_i] = \e_\l\e_i e_\l e_i$ for all $i$. For the
    	first relation we have
    	\[
    		\eta_{j,k,\l}(\e_\l)^2=
         (\e_j+\e_k-\theta\e_j\e_k)^2=\theta\e_j+\theta\e_k-2\e_j\e_k=\theta(\e_j+\e_k-\theta\e_j\e_k)=\eta_{j,k,\l}(\theta)\eta_{j,k,\l}(\e_\l).
      \]
    	
    	As for the second relation, for $i = \ell$ we have
    	\begin{eqnarray*}
    		\eta_{j,k,\l}(\e_\l)\eta_{j,k,\l}(\e_\l) \eta_{j,k,\l}(e_\l) \eta_{j,k,\l}(e_\l) & = & (\e_j+\e_k-\theta\e_j\e_k)^2 e_j e_k e_j e_k\\
    		& = & \theta(\e_j+\e_k-\theta\e_j\e_k)(1-\e_j\e_k)e_j e_j e_k e_k\\
    		& = & \theta(\e_j+\e_k-\theta\e_j\e_k)e_j^2 e_k^2=0=\eta_{j,k,\l}([e_l,e_l])
    	\end{eqnarray*}
    	since $\theta\e_j e_j^2=\theta\e_k e_k^2=0$. For $i \neq \l$,
    	\begin{eqnarray*}
    		\eta_{j,k,\l}(\e_\l)\eta_{j,k,\l}(\e_i)\eta_{j,k,\l}( e_\l)\eta_{j,k,\l}( e_i) & = & (\e_j+\e_k-\theta\e_j\e_k)\e_i e_j e_k e_i\\
    		& = & (\e_j\e_i+\e_k\e_i-\theta\e_j\e_k\e_i)e_j e_k e_i\\
    		& = & (1-(1-\e_j\e_i)(1-\e_k\e_i))e_j e_k e_i\\
    		& = & (1-\exp(\e_v\e_i))ve_i,
    	\end{eqnarray*}
    	where $v=e_j e_k$. But by \Pref{prop:exponent_property}, we know that $(1-\exp(\e_v\e_i))ve_i=ve_i-e_i v=[v,e_i]=
    	[\eta_{j,k,\l}(e_\l),\eta_{j,k,\l}(e_i)]=\eta_{j,k,\l}([e_\l,e_i])$, so $\eta_{j,k,\l}(\e_\l\e_i e_\l e_i)=\eta_{j,k,\l}([e_\l,e_i])$,
    	as we wanted to show.

    	Now compose the morphisms defined above so that each $e_i$ is
    	mapped to a word of length $\len(w_i)$ on distinct generators, and
    	then map the generators to the respective letters in the $w_i$.
    \end{proof}

    \begin{lem} \label{lem:e_i_are_generic}
    	Let $f(x_1,\dots,x_n)\in P_n$ be any multilinear polynomial in
    	non-commutative variables (with coefficients in $C$). Then $f\in\T{\G}$ iff
    	$f(e_1,\dots,e_n)=0$.
    \end{lem}

    \begin{proof}
    	If $f$ is an identity then obviously $f(e_1,\dots,e_n) = 0$. On
    	the other hand assume $f(e_1,\dots,e_n)= 0$. For every
    	$w_1,\dots,w_n$ we obtain $f(w_1,\dots,w_n) =
    	\eta_w(f(e_1,\dots,e_n)) = 0$, so we are done by multilinearity.
    \end{proof}

    \begin{proof}[of \Tref{thm:comodule_of_G}]
    	Let $M_n = P_n/({\T{\G} \cap P_n})$ denote the $n$-th co-module of
    	$\G$, where $P_n$ is defined in \eq{Pndef}. Define a linear
    	mapping $\mu\co M_n \rightarrow C[\e]_n e_1\cdots e_n$ by the
    	substitution $x_i\mapsto e_i$. By \Pref{prop:sign_properties},
    	$\displaystyle\sum_{\sigma\in S_n} a_\sigma x_{\sigma(1)}\cdots
    	x_{\sigma(n)}$ is mapped to $\displaystyle\sum_{\sigma\in S_n}
    	a_\sigma e_{\sigma(1)}\cdots
    	e_{\sigma(n)}=\displaystyle\sum_{\sigma\in S_n} a_\sigma
    	\esgn{w}{\sigma} e_1\cdots e_n$, where $w=(e_1,\dots,e_n)$. Let
    	$\nu \co C[\e]_n e_1\cdots e_n\rightarrow C[\e]_n$ denote the
    	isomorphism of $C$-modules defined by $\nu(\lambda e_1\cdots
    	e_n)=\lambda$. Let $\psi=\nu\circ\mu \co M_n\rightarrow C[\e]_n$.
    	We will prove that $\psi$ is an isomorphism of $S_n$-modules.

        Indeed, $\psi(\displaystyle\sum_{\sigma\in S_n} a_\sigma x_{\sigma(1)}\cdots x_{\sigma(n)})=\displaystyle
        \sum_{\sigma\in S_n} a_\sigma \esgn{w}{\sigma}$. But, for every $\pi \in S_n$,
        \begin{eqnarray*}
            \psi\pi(\displaystyle\sum_{\sigma\in S_n} a_\sigma x_{\sigma(1)}\cdots x_{\sigma(n)}) & = & \psi(\displaystyle\sum_{\sigma\in S_n}
                a_\sigma x_{\pi\sigma(1)}\cdots x_{\pi\sigma(n)})\\
            & = & \displaystyle\sum_{\sigma\in S_n} a_\sigma \esgn{w}{\pi\sigma}\\
            & = & \esgn{w}{\pi}\phi_\pi(\displaystyle\sum_{\sigma\in S_n} a_\sigma\esgn{w}{\sigma})\\
            & = & \esgn{w}{\pi}\phi_\pi(\psi(\displaystyle\sum_{\sigma\in S_n}a_\sigma x_{\sigma(1)}\cdots x_{\sigma(n)})) \\
            & = & \pi(\psi(\displaystyle\sum_{\sigma\in S_n}a_\sigma x_{\sigma(1)}\cdots x_{\sigma(n)})),
        \end{eqnarray*}
        showing that $\psi$ is a homomorphism of $S_n$-modules.

    Since $1 = \psi(x_1 \cdots x_n)$ generates $C[\e]_n$, $\psi$ is
    surjective. Injectivity follows once we show that if $f \in P_n$
    becomes zero under the substitution $x_i \mapsto e_i$ then $f$ is
    an identity, which is the content of \Lref{lem:e_i_are_generic}.
    \end{proof}

    In addition to having the co-modules of $\G$, we can already calculate its co-dimensions:

    \begin{thm} \label{thm:codim_of_G}
    The $S_n$-module $C[\e]_n$ is a free $C$-module of rank $2^{n-1}$.
    \end{thm}

    \begin{proof}
            In the proof of \Lref{lem:all_idents_of_G_are_consequences_of_grassmann}, we have seen that modulo consequences of the Grassmann
        identity, every non-commutative polynomial $f$ of degree $n$, and in particular every multilinear polynomial $f$ of degree
        $n$ is a sum of elements of the form $x_{i_1}\cdots x_{i_m}[x_{i_{m+1}},x_{i_{m+2}}]\cdots[x_{i_{n-1}},x_{i_n}]$
        where $i_1\leq\dots\leq i_m$ and we can assume that $i_{m+1}<\dots<i_n$. Therefore, they generate the $n$-th co-module
        of $\G$ as a $C$-module. Thus, if we let
        \begin{multline*}
            N=
                \text{span}_C\{x_{i_1}\cdots x_{i_m}[x_{i_{m+1}},x_{i_{m+2}}]\cdots[x_{i_{n-1}},x_{i_n}] \suchthat i_1<\dots<i_m, i_{m+1}<\dots<i_n\},
        \end{multline*}
        then $N/(N \cap \T{\G})$  is the $n$-th co-module of $\G$, which is (by
        \Tref{thm:comodule_of_G}) isomorphic to $C[\e]_n$. Hence, $C[\e]_n$ is the quotient of $N$ by all identities of $\G$. But, we have seen
        in the proof of \Lref{lem:all_idents_of_G_are_consequences_of_grassmann} that all identities of $\G$ in $N$ are zero, and hence
        $N$ is isomorphic to $C[\e]_n$.

        However, there are exactly $2^{n-1}$ polynomials in the set spanning $N$, and we have already seen that they are linearly independent: indeed, in the
        proof of \Lref{lem:all_idents_of_G_are_consequences_of_grassmann}, we have shown that if $\displaystyle\sum a_i x_{i_1}\cdots
        x_{i_m}[x_{i_{m+1}},x_{i_{m+2}}]\cdots[x_{i_{n-1}},x_{i_n}]\in N$ is an identity (in particular, a linear relation among the generators of $N$),
        then the coefficients $a_i$ are zero. Hence, they are linearly independent.
    \end{proof}

    \begin{cor}
    For any field $C=\F$ of any characteristic, the co-dimension
    sequence of $\G$ is $c_n(\G)=2^{n-1}$.
    \end{cor}

    An immediate result is that we know the co-dimension of $G$, the usual Grassmann algebra, for any field of
    characteristic different than $2$, generalizing the well known classical result in characteristic $0$ (see also, for a purely
    combinatoric proof, \cite{combinatorial_codim_of_G}).

    \begin{cor}
        For any field $\F$ with $\chr\F\neq 2$, we have $c_n(G)=2^{n-1}$.
    \end{cor}

    \begin{proof}
        We have shown that when $2$ is invertible, $\T{\freeG}=\T{\G}$ (see
        \Tref{thm:G_and_extended_G_equivalence}) -- and since $\freeG$ is an extension by scalars of $G$, they have the same co-dimension.
    \end{proof}

    \section{Generalized Superalgebras and Generalized Supertraces}
    \ifx\separatepaper\undefined
        \numberwithin{thm}{subsection}
    \fi

    \subsection{Generalized Superalgebras}

    Now that we have the basic machinery of the generalized Grassmann
    algebra, we would like to use it to replicate the success of the
    standard Grassmann algebra in characteristic $0$. The first
    problem is that while the Grassmann algebra $G$ has a natural
    superalgebra structure, given by the words of even and odd length,
    the even-odd grading on $\G$ is uninteresting, as exemplified by
    \Lref{lem:generization_of_G}.

    Recall the definition of $C[\e]$ in \Dref{def:C_epsilon}. Taking
    advantage of the many idempotents of $C[\e]$, we choose the
    following grading.

    \begin{dfn}
    	A $C[\e]$-algebra is called a \emph{\super{}algebra} over $C$ if it is graded  by the group $\Z_2^{\oplus
    	\N}=\bigoplus_{i\in\N}\Z_2$. 
    \end{dfn}

    Our first example is the algebra $\G$ itself:
    \begin{dfn}
    	The extended Grassmann algebra is $\Z_2^{\oplus\N}$-graded by
    	letting $C[\e]$ be contained in the zero component, and setting
    	the grade of each $e_i$ to be $(0,\dots,0,1,0,\dots)$ where the
    	$1$ is in the $i$-th component. The degree of a word $w \in \G$ is
    	$g=(\deg_1 w,\deg_2 w,\dots)$ modulo $2$, where $\deg_i w$ is the
    	number of occurrences of $e_i$ in $w$.
    \end{dfn}

    The zero component is thus $\G_0 = C[\e][e_1^2,e_2^2,\dots]$, which is contained in the center of $\G$. For every
    $g=(g_1,g_2,\dots)\in\Z_2^{\oplus \N}$, let $e_g=\prod e_i^{g_i}$ and $\e_g=g_1\e_1+g_2\e_2\cdots\in\text{span}_{\Z_2}\{\e_i\}$,
    which are finite products and sums. The corresponding component $\G_g = \G_0 e_g$ is a rank $1$ module over $\G_0$, so the grading
    is ``thin''.

    \begin{dfn} \label{def:ext_super_commutator}
    Let $\eA=\bigoplus_{g\in\Z_2^{\oplus\N}}\eA_g$ be any
    \super{}algebra over $C$. We define the \emph{\super{}commutator}
    $\{a,b\}\in\eA$ for homogenous elements $a\in \eA_g$, $ b\in
    \eA_h$ by setting $\{a,b\}=ab-\exp{(\e_g\e_h)}ba$, extended
    bilinearly to all $a,b\in\eA$.

    We say that $\eA$ is \emph{\super{}commutative} if $\{a,b\}=0$
    for all $a,b\in \eA$.
    \end{dfn}

    \begin{exm}
    The extended Grassmann algebra $\G$ is \super{}commutative.
    Indeed, by \Pref{prop:exponent_property}, for any pair of words $u
    \in \G_g$ and $v \in \G_h$ we have
    $uv=\exp(\e_{g}\e_{h})vu=\exp(\e_g\e_h)vu$, or in other words,
    $\{u,v\}=0$.
    \end{exm}

    We will use regular font for the standard supertheoretic notions,
    such as $\sgn{\cdot}$, $\sCent$, $\str$, $A$, $B$, $C$, $G$, and the Fraktur
     font for the corresponding \super{}theory notions,
    $\esgn{}{\cdot}$, $\esCent$, $\estr$, $\eA$, $\eB$, $\eC$, $\G$,
    etc.

    \begin{exm}\label{exm:efreeG}
    As another example, one can consider $\efreeG$, the free
    \super{}commutative \super{}algebra on the generators $e_g^{(n)}$
    ($n = 1,2,\dots$)
        where $e_g^{(n)}\in\efreeG_g$ is a homogenous generator of the component with degree $g$. As a result, $\efreeG$ is generated by the
        generators $e_g^{(n)}$ under the relations:
        \begin{equation*}
            [e_g^{(n)},e_h^{(m)}]=(1-\exp(\e_g\e_h))e_g^{(n)}e_h^{(m)}.
        \end{equation*}
        Note that $\T{\G}=\T{\efreeG}$, because $\G\subset\efreeG$ and $\efreeG$ satisfies the Grassmann identity.
    \end{exm}

    \subsection{The Generalized Grassmann Hull}

    Now that we have an appropriate grading, we can generalize the
    Grassmann hull of an algebra (see
    \Tref{thm:kemer_superrepresentability} for the notion of the
    Grassmann hull for superalgebras). Similarly to the standard
    Grassmann hull, one can use either the Grassmann algebra or the
    free \super{}commutative algebra to define it (for an
    example in the case of $\chr=0$, see
    \cite[p.~83--85]{asymptotic_methods}). For our purposes, it will be
    more convenient to use the free \super{}commutative algebra.

    \begin{dfn}
    Let $\eA=\bigoplus_{g\in\Z_2^{\oplus\N}}\eA_g$ be a
    \super{}algebra. The generalized Grassmann hull is by definition
        \[
             \efreeG[\eA]=\bigoplus_{g\in\Z_2^{\oplus\N}}\left(\efreeG_g\otimes_{C[\e]}\eA_g\right),
        \]
    with the $\Z_2^{\oplus\N}$-grading defined by
    $\efreeG[\eA]_g=\efreeG_g\otimes_{C[\e]}\eA_g$.
    \end{dfn}

    \begin{exm}
    Let $A$ be any $C$-algebra. Tensoring with the $C[\e]$-group
    algebra $C[\e][\Z_2^{\oplus\N}]$, which is naturally a \super{}algebra over $C$, gives $A\otimes_C
    C[\e][\Z_2^{\oplus\N}]$ a natural \super{}algebra grading, where
    $(A\otimes_C C[\e][\Z_2^{\oplus\N}])_g = A \otimes_C
    (C[\e][\Z_2^{\oplus\N}])_g$ and
        \[
            A\otimes_C \efreeG=\efreeG[A\otimes_C C[\e][\Z_2^{\oplus\N}]].
        \]
    \end{exm}

    We will now define the notion of a \super{}identity:

    \begin{dfn}\label{def:sigma_id}
        Define $C[\e]\sg{x^{(g)}_1,x^{(g)}_2,\dots\suchthat g\in\Z_2^{\oplus\N}}$ to be the free \super{}algebra.
        The elements of this algebra, which is denoted by $C[\e]\sg{X^{(g)}}$ for brevity, are called \super{}polynomials.
        We will define the set of \super{}identities of any \super{}algebra $\eA$ as the intersection of all kernels of all
        grading-preserving $C[\e]$-homomorphisms $\phi:C[\e]\sg{X^{(g)}}\rightarrow \eA$, and denote it by $\T[\Sigma]{\eA}$.
    \end{dfn}

    \begin{dfn}
        For every finitely supported function $\bar{n} \co \Z_{2}^{\oplus \N} \rightarrow \N$, $g \mapsto \bar{n}^{(g)}$,
        we let $P_{\bar{n}}[\e]$ denote the $C[\e]$-module of multilinear \super{}polynomials with coefficients in
        $C[\e]$, in the variables $\{x_i^{(g)}\}_{1\leq i\leq \bar{n}^{(g)},g\in\Z_2^{\oplus\N}}$. We will refer to $\bar{n}$ as
        the associated \emph{multidegree}. We will also write $n=\sum\bar{n}^{(g)}$, the total degree of identities in $P_{\bar{n}}[\e]$.
        The multilinear part of $C[\e]\sg{X^{(g)}}$ is $\bigoplus_{\bar{n}} P_{\bar{n}}[\e]$.
    \end{dfn}

    Again, keeping the analogy to the case of characteristic $0$, we can define the operation of the generalized Grassmann
    hull on an identity.

    \begin{dfn}
    We define the \emph{Grassmann involution} on \super{}polynomials
    as follows.  Let $f=\displaystyle\sum_{\sigma\in S_n}a_\sigma
    x_{\sigma(1)}\cdots x_{\sigma(n)}\in P_{\bar{n}}[\e]$ be a
    multilinear $\Z_2^{\oplus\N}$-graded identity of multidegree
    $\bar{n}$, such that each variable $x_j$ is in the homogenous
    component of $C[\e]\sg{X^{(g)}}$ corresponding to $g_j$. Then
    $$f^*=\displaystyle\sum_{\sigma\in S_n}\esgn{w}{\sigma}a_\sigma x_{\sigma(1)}\cdots x_{\sigma(n)},$$
    where $w=(e_{g_1},\dots,e_{g_n})$.

        (Although $w$ is not well defined, the $\e$-counterpart $\e_{g_1},\dots,\e_{g_n}$ is well defined. Hence, since $\esgn{w}{\sigma}$
        only depends on the $\e_{w_i}$, the morphism $*$ is well
        defined.)
    \end{dfn}

    This is indeed an involution:
    \begin{lem} \label{lem:Grass_hull_is_an_inv}
    The map $f \mapsto f^*$ is an involution.
    \end{lem}
    \begin{proof}
    Let $f\in P_{\bar{n}}[\e]$ be a multilinear
    $\Z_2^{\oplus\N}$-graded identity of multidegree $\bar{n}$. Write
    $f=\displaystyle\sum_{\sigma\in S_n}a_\sigma x_{\sigma(1)}\cdots
    x_{\sigma(n)}$. Then $f^{**}=\displaystyle\sum_{\sigma\in S_n}
        \esgn{w}{\sigma}^2 a_\sigma x_{\sigma(1)}\cdots x_{\sigma(n)}$. But,
        \begin{eqnarray*}
            \esgn{w}{\sigma}^2 = \exp\left(\displaystyle\sum_{\stackrel{i<j}{\sigma(i)>\sigma(j)}}\e_{w_{\sigma(i)}}\e_{w_{\sigma(j)}}\right)^2 = 1
        \end{eqnarray*}
    by \Rref{rem:exp2}, so $f^{**} = f$.
    \end{proof}

    As is the case with superalgebras, the involution gives the identities of the
    generalized Grassmann hull:
    \begin{dfn}
        Let $\Gamma\triangleleft C[\e]\sg{X^{(g)}}$ be a two-sided ideal. We say that $\Gamma$ is a T$_\Sigma$-ideal if it is also
        invariant under all $C[\e]$-endomorphisms of $C[\e]\sg{X^{(g)}}$ that preserve the grading.

        Also, in this case, we let $\Gamma^*$ be the T$_\Sigma$-ideal generated as a T$_\Sigma$-ideal by the images of all multilinear
        identities in $\Gamma$ under the involution $*$.
    \end{dfn}

    \begin{rem}
        Note that for all T$_\Sigma$-ideals $\Gamma$, we have: $\Gamma^*\cap P_{\bar{n}}[\e]=(\Gamma\cap P_{\bar{n}}[\e])^*$, where
        on the right hand side, taking $*$ means taking $*$ on each element separately. This is because the multilinear part
        $\Gamma\cap P_{\bar{n}}[\e]$ is already endomorphism-invariant, and since, by definition, $\Gamma^*$ is the minimal
        T$_\Sigma$-ideal containing $(\Gamma\cap P_{\bar{n}}[\e])^*$.

        In other words, using $*$ on all multilinear identities of a T$_\Sigma$-ideal $\Gamma$ gives all multilinear identities of
        $\Gamma^*$.
    \end{rem}

    Recall that $\T[\Sigma]{\eA}$ is the set of $\Sigma$-identities of $\eA$, \Dref{def:sigma_id}.

    \begin{thm} \label{thm:idents_of_Grass_hull}
        Let $A$ be a \super{}algebra. Then $\T[\Sigma]{\efreeG[\eA]}$ and $\T[\Sigma]{\eA}^*$ have the same multilinear components.

        In other words, for every $f\in P_{\bar{n}}[\e]$, we have that $f\in\T[\Sigma]{\efreeG[\eA]}$ iff $f^*\in\T[\Sigma]{\eA}$.
    \end{thm}

    \begin{proof}
        Let $f=\displaystyle\sum_{\sigma\in S_n}\alpha_\sigma x_{\sigma(1)}\cdots x_{\sigma(n)}\in P_{\bar{n}}[\e]$. Let $x_i\mapsto a_i\otimes w_i$
        be any substitution where $w_i\in\efreeG_{g_i}$ is a word in the generators $e_j^{(n)}$ of $\efreeG$, in the component corresponding to $g_i$,
        and $a_i\in \eA_{g_i}$. Then, under the substitution:
        \begin{eqnarray*}
            f & \mapsto & \displaystyle\sum_{\sigma\in S_n}\alpha_\sigma a_{\sigma(1)}\otimes w_{\sigma(1)}\cdots a_{\sigma(n)}\otimes w_{\sigma(n)}\\
            & = & \displaystyle\sum_{\sigma\in S_n}\alpha_\sigma (a_{\sigma(1)}\cdots a_{\sigma(n)})\otimes(w_{\sigma(1)}\cdots w_{\sigma(n)})\\
            & \stackrel{\mbox{Prop.~\ref{prop:sign_properties}}}{=} & \displaystyle\sum_{\sigma\in S_n}\alpha_\sigma (a_{\sigma(1)}\cdots a_{\sigma(n)})\otimes(\esgn{w}{\sigma}w_1\cdots w_n)\\
            & = & \displaystyle\sum_{\sigma\in S_n}\alpha_\sigma(\esgn{w}{\sigma}a_{\sigma(1)}\cdots a_{\sigma(n)})\otimes(w_1\cdots w_n)\\
            & = & \(\displaystyle\sum_{\sigma\in S_n}\alpha_\sigma\esgn{w}{\sigma}a_{\sigma(1)}\cdots a_{\sigma(n)}\)\otimes(w_1\cdots w_n)\\
                & = & f^*(a_1,\dots,a_n)\otimes w_1 \cdots w_n,
        \end{eqnarray*}
        as we wanted to show.
    \end{proof}

    \begin{cor}
        Let $\eA$ be a \super{}algebra. Then: $\T[\Sigma]{\efreeG[\efreeG[\eA]]}$ and $\T[\Sigma]{\eA}$ share the same multilinear identities.
    \end{cor}

    \begin{proof}
        Use the result of \Tref{thm:idents_of_Grass_hull} twice, and then apply \Lref{lem:Grass_hull_is_an_inv}.
    \end{proof}

    \begin{rem}
    We have \emph{not} proved that
    $\T[\Sigma]{\efreeG[\eA]}=\T[\Sigma]{\eA}^*$. In characteristic $0$,
    having the same multilinear identities would have implied that
    they are the same. However, this is not the case in positive
    characteristic: $\T[\Sigma]{\eA}$ is not necessarily generated as
    a T$_{\Sigma}$-ideal by its multilinear component.
    \end{rem}

    We see that even though the language of generalized Grassmann
    hulls generalizes the ordinary notion of Grassmann hull, its
    formulation could be considered more elegant; rather than defining
    the involution on a multilinear identity by multiplying by the
    sign of only the odd variables, we simply multiply by the
    generalized sign of all variables. This is mainly because all
    words in the generators $e_i$ of $\G$ are, in a way, generic, so
    there is no ``special treatment" of any specific component of the
    grading.

    \subsection{Generalized Supertraces}

    The superization of basic concepts in linear algebra, such as the
    supertrace and supercommutator, is defined in characteristic zero.
    We now begin the development of a supertheory based upon $\G$ and
    the concept of the generalized superalgebra. Such a \super{}theory will have the advantage of
    being characteristic free, valid over any ring.

    We will begin by defining the notion of \super{}traces. Recall that a trace function on a $C$-algebra $A$ is a function $\tr\co
    A\ra\Cent(A)$ satisfying $\tr[a,b]=0$ and $\tr(a\tr(b))=\tr(a)\tr(b)$.

    \begin{dfn}
    Let $\eA$ be a \super{}algebra over $C$. Its \emph{\super{}center},
    $\esCent{(\eA)}$, is the set of all elements of $\eA$ that
    \super{}commute with every element, i.e.
        \[
            \esCent{(\eA)}=\{a\in \eA\suchthat\forall{b\in \eA},\{a,b\}=0\},
        \]
        where
        $\{a,b\}$ is the \super{}commutator of \Dref{def:ext_super_commutator}.
    \end{dfn}

    \begin{dfn} \label{def:ext_supertrace}
       Let $\eA$ be a \super{}algebra over $C$. A $C[\e]$-linear (grading-preserving) function $\estr:\eA\rightarrow\esCent{(\eA)}$
        will be called a \emph{\super{}trace} iff
        \[
            \estr{\{a,b\}}=0
        \]
        and
        \[
            \estr{(a\,\estr{(b)})}=\estr{(a)}\estr{(b)},
        \]
        for every $a,b\in \eA$.
    \end{dfn}

    The concepts of \emph{\super{}trace \super{}identities} naturally
    follows (see \cite[chapter~12]{comp_aspects}):

    \begin{dfn}
    Define the algebra $C[\e]\sg{X^{(g)},\esTr}$ to be the free
    \super{}algebra with \super{}trace $\esTr$. This algebra is spanned over $C[\e]$ by words of the form $w_0 \esTr(w_1) \cdots \esTr(w_\ell)$ where $w_i \in \sg{X^{(g)}}$, and the grading is such that the grade of $\esTr(w)$ is the same as that of $w$. The defining relations are the axioms of \Dref{def:ext_supertrace}.

    The \super{}trace \super{}identities of a \super{}algebra $\eA$
    with \super{}trace $\estr$ are the elements in the  intersection
    of all the kernels of all grading-preserving $C[\e]$-homomorphisms
    $\phi\co C[\e]\sg{X^{(g)},\esTr}\rightarrow A$    such that
    $\estr{\phi(x)}=\phi(\esTr{\,x})$.
    \end{dfn}

    \begin{rem} \label{rem:formal_trace_notation}
    We use different capitaliztion to differentiate between formal
    traces (traces in the free algebra) and traces of the object under
    discussion. That is, $\Tr$, $\sTr$ and $\esTr$ are formal traces,
    formal supertraces and formal \super{}traces in the algebras
    $C\sg{X,\Tr}$, $C\sg{X^{(0)},X^{(1)},\sTr}$ and
    $C[\e]\sg{X^{(g)},\esTr}$, respectively. At the same time, $\tr$,
    $\str$ and $\estr$ are arbitrary trace functions, in any algebra
    we happen to be currently working with.
    \end{rem}

    For example, the equality $\estr(a^p)=\estr(a)^p$ holds in the
    algebra $A$ for all $a$, iff $A$ satisfies the \super{}trace
    \super{}identity $\esTr(x^p)=\esTr(x)^p$. In other words,
    $\esTr(x^p)=\esTr(x)^p$ is an identity, while
     $\estr(a^p)=\estr(a)^p$ is the value of that identity after substituting the function $\estr$ to the variable $\esTr$.

    We come to our most important example.

    \begin{dfn} \label{def:assoc_ext_supertrace}
    Let $\eA$ be a \super{}algebra with a grading preserving trace
    function $\tr \co \eA \ra C$. Define the associated \super{}trace
    function $\estr=\tr^*$ on $\efreeG[\eA]$ by $\estr(a\otimes
    w)=\tr(a)\otimes w$.

        Conversely, if $\eA$ has a \super{}trace $\estr$, define its associated trace function $\tr=\estr^*$ on
        $\efreeG[\eA]$ by $\tr(a\otimes w)=\estr(a)\otimes w$. Note that $\estr^*$ preserves the grading.
    \end{dfn}

    \begin{lem}
        The above definitions of the associated trace function $\estr^*$ and the associated \super{}trace function
        $\tr^*$, indeed give a trace function and a \super{}trace function, respectively.
    \end{lem}

    \begin{proof}
    	This follows since for all $a\otimes u,b\otimes v\in\efreeG[\eA]$, $a,b\in \eA$,
    	$u,v\in\efreeG$,
    	\begin{eqnarray*}
    		\{a\otimes u,b\otimes v\} & = & (a\otimes u)(b\otimes v) - \exp(\e_g\e_h)(b\otimes v)(a\otimes u) \\
    		& = & ab \otimes uv -ba \otimes \exp(\e_g\e_h) vu \\
    		& = & (ab -ba)\otimes uv = [a,b]\otimes uv
    	\end{eqnarray*}
    	and $\{a,b\}\otimes uv=[a\otimes u,b\otimes v]$ in the same manner.
    \end{proof}

    \begin{rem}
        Let $A$ be a $C$-algebra, with a trace function $\tr$. Then $A\otimes_C C[\e][\Z_2^{\oplus\N}]$ has a \super{}algebra
        grading (coming from the grading of the $C[\e][\Z_2^{\oplus\N}]$ component), and is a \super{}algebra.
        Now, the function $\tr\co A\rightarrow\Cent(A)$ can be extended by linearity to $\hat{\tr} = \tr \otimes 1 \co A\otimes_C C[\e][\Z_2^{\oplus\N}]
        \rightarrow\Cent(A)\otimes C[\e][\Z_2^{\oplus\N}]$ such that $\hat{\tr}$ preserves the \super{}algebra grading.
        Then, since $A\otimes_C \efreeG=\efreeG[A\otimes_C C[\e][\Z_2^{\oplus\N}]]$, we obtain a \super{}trace $\estr=\hat{\tr}^*$
        on $A\otimes_C \efreeG$, given by: $\estr(a\otimes w)=\tr(a)\otimes w$. This construction generalizes \Dref{def:assoc_ext_supertrace}
        to the case of (non-graded) $C$-algebras.
    \end{rem}

    Now, in analogue with \Tref{thm:G_and_extended_G_equivalence}, we show the equivalence of supertrace and
    \super{}strace identities (the identities are not graded, so these are not \super{}identities).

    \begin{thm} \label{thm:supertrace_and_extended_supertrace_equivalence}
    	\ThmTraceEquiv.
    \end{thm}

    \begin{proof}
    	The proof is virtually identical, word for word, to the proof of
    	\Tref{thm:G_and_extended_G_equivalence}.
    \end{proof}

    A key result in PI-theory is the ``Kemer supertrick'' (see {\it{e.g.}} \cite{Zel}), which heavily relies on representation
    theory, which fails to deliver in characteristic $p$. The Kemer supertrick can be reformulated as the claim that for all algebras
    $A$ there is some $n$ such that $\T{A}\supseteq\T{\M[n](G)}$. In this sense, the Kemer supertrick has already been proven in
    characteristic $p$ (by Kemer, \cite{kemer_basis_of_mult_trace_idents_of_matrices}), but with very bad bounds.

    On the long term, one might hope to bypass this difficulty by directly adding formal supertraces to algebras (and then show that
    their identities imply all identities of $\M[n](\G)$), just like Zubrilin theory (see \cite{specht_for_graded_algebras} for an
    overview of Zubrillin traces) enables the introduction of traces to an algebra and showing that affine PI-algebras satisfy all
    identities of a matrix algebra.

    This motivates the following question about \super{}traces: Let $A$ be an (ordinary) algebra on which a linear function $\ef$ is
    defined. What identities on $A$ and $\ef$ allow us to introduce a grading to the algebra such that $\ef$ becomes a \super{}trace?

    More formally, we define

    \begin{dfn}
        Let $C\sg{X,\eF}$ be the free algebra over $C$ with a $C$-linear function $\eF$ acting freely on it. Let $A$ be any
        $C$-algebra with a linear function $\ef:A\rightarrow A$. We define the \emph{identities of $A$ with linear function $\ef$}
        to be the intersection of a;; kernels of all homomorphisms $\phi:C\sg{X,\eF}\rightarrow A$ such that
        $\phi(\eF{(a)})=\ef{(\phi(a))}$.
    \end{dfn}

    \begin{rem}
        As in \Rref{rem:formal_trace_notation}, we use capitalization to differentiate formal objects from
        others. That is, $\ef$ is any particular linear function, while $\eF$ is \emph{the} formal linear function, of the algebra
        $C\sg{X,\eF}$.
    \end{rem}

    \begin{thm} \label{thm:basis_of_identities_of_free_str}
        The multilinear part of the ideal of identities of $C[\e]\sg{X^{(g)},\esTr}$ with linear function $\esTr$ is generated by:
        \begin{eqnarray*}
            \eF(\eF(x)y)& = & \eF(x)\eF(y)\\
            \eF(x\eF(y))& = & \eF(x)\eF(y)\\
            {}[x,\eF{[y,z]}] & = & 0\\
            {}[\eF(x),[\eF(y),z]] & = & 0
        \end{eqnarray*}
    \end{thm}

    Note that the \super{}identity $\eF{\{a,b\}}=0$ of \Dref{def:ext_supertrace} is not in the list, as it is not an (ordinary) identity.

    To prove the theorem we require a few lemmas. We begin by proving a lemma analogous to \Lref{lem:G_satisfies_grass}:

    \begin{lem} \label{lem:free_str_satisfies_grassmann}
        The identities with linear function:
        \begin{subequations}
            \begin{eqnarray}
                \eF(\eF(x)y)& = & \eF(x)\eF(y)  \label{eq:str_swallow_right}\\
                \eF(x\eF(y))& = & \eF(x)\eF(y)  \label{eq:str_swallow_left}\\
                {}[x,\eF{[y,z]}] & = & 0    \label{eq:str_of_commutator_is_central}\\
                {}[\eF(x),[\eF(y),z]] & = & 0    \label{eq:str_commutes_with_commutator_with_str}
            \end{eqnarray}
        \end{subequations}
        hold in $C[\e]\sg{X^{(g)},\esTr}$.
    \end{lem}

    \begin{proof}
        The identities \Eqref{eq:str_swallow_right} and \Eqref{eq:str_swallow_left} follow immediately from the definition
        of the \super{}trace (\Dref{def:ext_supertrace}).

        We will now show that the identities \Eqref{eq:str_of_commutator_is_central} and \Eqref{eq:str_commutes_with_commutator_with_str}
        are indeed satisfied by any \super{}trace, using the fact that the \super{}traces \super{}commute with everything
        and a product of two elements inside a \super{}trace behaves as if it \super{}commutes. Thus, for the purpose of checking
        \Eqref{eq:str_of_commutator_is_central} and \Eqref{eq:str_commutes_with_commutator_with_str}, one can assume that everything
        \super{}commutes. But the \super{}commutative \super{}algebra $\G$ satisfies the Grassmann identity, which thus implies
        these two identities.

        More formally, we begin by proving \Eqref{eq:str_of_commutator_is_central}. The proof of \Eqref{eq:str_commutes_with_commutator_with_str}
        is completely analogous. First of all, since \Eqref{eq:str_of_commutator_is_central} is multilinear, we may assume that
        $x$, $y$ and $z$ are all homogenous. Then the following holds:
        \begin{eqnarray*}
            [x,\esTr{[y,z]}] & = & [x,\esTr(yz)-\esTr(zy)]=[x,\esTr(yz)-\exp(\e_y\e_z)\esTr(yz)]\\
            & = & (1-\exp(\e_y\e_z))[x,\esTr(yz)]\\
            & = & (1-\exp(\e_y\e_z))(x\,\esTr(yz)-\esTr(yz)x)\\
            & = & (1-\exp(\e_y\e_z))(x\,\esTr(yz)-\exp(\e_x\e_{yz})x\,\esTr(yz))\\
            & = & (1-\exp(\e_y\e_z))(1-\exp(\e_x(\e_y+\e_z)))x\,\esTr(yz).
        \end{eqnarray*}
        Hence, in order to show that this is zero, it is sufficient to show that $0=(1-\exp(\e_y\e_z))(1-\exp(\e_x(\e_y+\e_z)))$.

        However, if we choose words $w_x,w_y,w_z\in\G$ such that $\e_{w_x}=\e_x$, $\e_{w_y}=\e_y$ and $\e_{w_z}=\e_z$,
        then since $\G$ satisfies the Grassmann identity:
        \begin{eqnarray*}
            0=[w_x,[w_y,w_z]] & = & [w_x,w_y w_z-w_z w_y]=[w_x,w_y w_z-\exp(\e_y\e_z)w_y w_z]\\
            & = & (1-\exp(\e_y\e_z))[w_x,w_y w_z]\\
            & = & (1-\exp(\e_y\e_z))(w_x w_y w_z-w_y w_z w_x)\\
            & = & (1-\exp(\e_y\e_z))(w_x w_y w_z-\exp(\e_x\e_{yz})w_x w_y w_z)\\
            & = & (1-\exp(\e_y\e_z))(1-\exp(\e_x(\e_y+\e_z)))w_x w_y w_z,
        \end{eqnarray*}
        and thus $(1-\exp(\e_y\e_z))(1-\exp(\e_x(\e_y+\e_z)))=0$, as we wanted to show.
    \end{proof}

    Now, we give the lemma analogous to \Lref{lem:G_satisfies_grass_consequ}:

    \begin{lem} \label{lem:free_str_satisfies_grassmann_consequ}
        The identities with linear function:
        \begin{subequations}
            \begin{eqnarray}
                [x,[\eF(y),\eF(z)]] & = & 0 \label{eq:commutator_of_str_is_central}\\
                {}[x,\eF(y)][\eF(z),\eF(w)]+[x,\eF(z)][\eF(y),\eF(w)] & = & 0 \label{eq:grass_consequ_str_moves}\\
                {}[\eF(x),y][\eF(z),\eF(w)]+[\eF(x),\eF(z)][y,\eF(w)] & = & 0 \label{eq:grass_consequ_var_moves}
            \end{eqnarray}
        \end{subequations}
        are consequences of \Eqref{eq:str_swallow_right}, \Eqref{eq:str_swallow_left}, \Eqref{eq:str_of_commutator_is_central}
        and \Eqref{eq:str_commutes_with_commutator_with_str}.
    \end{lem}

    \begin{proof}
    In order to obtain \Eqref{eq:commutator_of_str_is_central}, substitute $y\mapsto\eF(y)$ into \Eqref{eq:str_of_commutator_is_central},
        and use \Eqref{eq:str_swallow_right} and \Eqref{eq:str_swallow_left} to see that $0=[x,\eF{[y,z]}]\mapsto
        [x,\eF{[\eF(y),z]}]=[x,[\eF(y),\eF(z)]]$.
        The proofs of \Eqref{eq:grass_consequ_str_moves} and \Eqref{eq:grass_consequ_var_moves}, given
        \Eqref{eq:str_commutes_with_commutator_with_str} and \Eqref{eq:commutator_of_str_is_central} are completely analogous to
        the proof of \Lref{lem:G_satisfies_grass_consequ}.
    \end{proof}

    \begin{proof}[of \Tref{thm:basis_of_identities_of_free_str}]
        We will use the following equalitites: for all words $s=x_1\cdots x_n$ and $t$, we have
        \begin{subequations}
            \begin{eqnarray}
                {}[s,t] & = & [x_1,x_2\cdots x_n t]+[x_2,x_3\cdots x_n t x_1]+\dots+[x_n,t x_1 x_2\cdots x_{n-1}], \label{eq:asymm_cyclic_trace_equality}\\
                {}0 & = & [x_1,x_2\cdots x_n]+[x_2,x_3\cdots x_n x_1]+\dots+[x_n,x_1 x_2\cdots x_{n-1}]. \label{eq:cyclic_trace_equality}
            \end{eqnarray}
        \end{subequations}

        The strategy of our proof greatly resembles that of \Lref{lem:all_idents_of_G_are_consequences_of_grassmann}. We
        will use the above identities to bring an arbitrary polynomial $f\in C\sg{X,\eF}$ to a specified standard
        form, and then use substitutions to show that the coefficients are $0$. This will be done via substitutions
        from the matrix algebras $\M[n](\G)$ over $\G$, with the \super{}traces $\estr$ associated with the usual
        traces in $\M[n](C)$.

        We begin by specifying the standard form we will use. Note that we are working with multilinear polynomials.
        The form is a sum of terms of the form:
    		\[\begin{array}{c}
            w\times\\
            \eF(v_1)\eF(v_2)\cdots\eF(v_n)\times\\
            {}[w_1,\eF(u_1)][w_2,\eF(u_2)]\cdots[w_m,\eF(u_m)]\times\\
            {}[\eF(u_{m+1}),\eF(u_{m+2})][\eF(u_{m+3}),\eF(u_{m+4})]\cdots[\eF(u_{k-1}),\eF(u_k)]\times\\
            \eF{[s_1,t_1]}\eF{[s_2,t_2]}\cdots\eF{[s_\l,t_\l]}
        \end{array}\]
        where $w$, $w_1,\dots,w_m$, $v_1,\dots,v_n$, $u_1,\dots,u_k$ and $t_1,\dots,t_\l$ are
        all words in the $x_i$, and the $s_1,\dots,s_\l$ are letters. However, many of these forms are trivially
        equal, so we require that: the words $u_1,\dots,u_k$ are alphabetically ordered; the words
        $v_1,\dots,v_n$ are alphabetically ordered; the pairs $(s_i,t_i)$ are also alphabetically ordered;
    		for every $i$, the letter $s_i$ is smaller than some letter of $t_i$; and the words $v_i$ and $u_i$ are cyclically
    		minimal, where a word is cyclically minimal if it is the first among its cyclic rotations.

        \Lref{lem:free_str_satisfies_grassmann}, \Lref{lem:free_str_satisfies_grassmann_consequ}
        and \Eqref{eq:asymm_cyclic_trace_equality}, \Eqref{eq:cyclic_trace_equality} imply that all multilinear polynomials
        can be brought to this form.

        Now, we will show that the coefficients of the terms containing no $\eF$-s are zero. Indeed, substitute
        matrix units $x_i\mapsto e_{\sigma(i),\sigma(i+1)}$ into all $x_i$, where $\sigma$ is some permutation.
        Then only the monomial in which the $x_i$ are ordered according to $\sigma$ contributes, and thus its
        coefficient is $0$.

        Next, rather than substitute a path as we just did, we choose some subset of the variables
        and substitute a cycle into them and a path into the rest. Since the standard trace is zero off diagonally, the only terms contributing are those
        that have no more than one appearance of $\eF$, corresponding to the cycle. We thus have three options for the terms that contribute:
        $w\cdot\eF(v_1)$, $w\cdot[w_1,\eF(u_1)]$ and $w\cdot\eF{[s_1,t_1]}$.

        Note that the last two do not contribute at all if the coefficients of the matrix units are central. Thus
        the coefficient of the first is $0$. Now, substitute coefficients from $\G$
        to two edges of the loop, such that exactly one edge has $e_1$ as the coefficient, and another
        has $e_2$ as the coefficient. Then only the term $w\cdot\eF{[s_1,t_1]}$ contributes -- and hence has coefficient equal to $0$.
        Finally, substitute $e_1$ to just one of the variables of the loop, and $e_2$ to an edge of the path.
        Then the term $w\cdot[w_1,\eF(u_1)]$ gives a non-zero contribution, unless it too has coefficient zero.

        We use induction on $N=n+k+\l$ to show that all coefficients are $0$. We substitute matrix elements such that there is one path, and $N=n+k+\l$ loops.
        We are now left with the liberty to choose their coefficients from $\G$. Now, we must be able to tell
        how they are divided into $v_i$-s, $u_i$-s and $(s_i,t_i)$-s. So, at first we substitute only central coefficients.
        This gives us the case of: $k=\l=0$, so its coefficient is zero.

        Now, we will use induction on $k+\l$. We choose $n=N-(k+\l)$ loops, and substitute central coefficients.
        This forces them to be $v_1,\dots,v_n$, and by induction, no coefficient with any other \mbox{$v_i$-s} contributes.
        Now, we substitute coefficients $e_i$ into all elements of the path, and we substitute one coefficient
        into the generators in each remaining loop (out of the $k+\l$ loops left). This gives us the case where
        $\l=0$.

        We use induction on $\l$. Choose $k$ loops, and substitute one coefficient into each one of them, in addition
        to the substitution into elements of the path. This forces these loops to be the $u_1,\dots,u_k$. We are left
        with two things to find out: how is the path split into the $w,w_1,\dots,w_m$, and how are the remaining
        $\l$ loops divided between the $s_i$ and the $t_i$.

        Choosing the division of each remaining loop into $s_i$ and $t_i$ is easy, and will be done via induction
        on the position of the letter $s_i$ relative to the largest letter of $t_i$. Indeed, the base of the
        induction is this: substitute a coefficient $e_{i_1}$ to the largest letter and  $e_{i_2}$ to the letter
        before it. Then the only contribution to the coefficient of the product $\e_{i_1}\e_{i_2}$ comes from the cases
        in which the largest letter itself is $s_i$, or the one before it is $s_i$ (otherwise $e_{i_1}$ and $e_{i_2}$ appear in
        their correct order). But because the largest letter is never $s_i$, we see that $s_i$ is also never
        the letter before that. Proceeding by induction, we are done.

        Therefore, we have almost isolated all coefficients of the form; we must now isolate one specific way to
        break down the path to $w,w_1,\dots,w_m$, for an arbitrary (but known) choice of $s_i$. This is done as
        follows. We use induction on $m$. Now, we already know that the associated loop, $u_i$, has one
        coefficient, say $e_1$, and we know which loop it is. Also recall that we substituted
        coefficients into the elements of the path. So, after the substitution, look for the largest
        number of $\e_i$-s appearing. This information determines which elements of the path belong to $w$ (their $\e_i$-s
        never appear). Now look for the smallest number of $\e_i$-s from the path appearing. This is the case
        where each $w_i$ contributes one $\e_i$. So, sort these $\e_i$-s, and put the element of the path
        corresponding to the $j$-th $\e_i$ into $w_j$. This gives us all elements of $w_j$, and only the case
        where $m$ is the smallest value we have not considered, contributes.

        This isolates everything -- only one term contributes, and thus has a coefficient of zero, which completes
        all of the above inductive steps.
    \end{proof}

    Note that incidently, just like in \Lref{lem:all_idents_of_G_are_consequences_of_grassmann}, we also obtain
    the co-dimension sequence (the algebra $C[\e]\sg{X^{(g)},\esTr}$ is not PI, so it is not exponential and also
    not very interesting).

    \begin{cor}
        Suppose that $A$ is any $C$-algebra, and $\ef$ any linear function on it. Also suppose that the following
        is true in $A$:
        \begin{eqnarray*}
            \ef(\ef(x)y)& = & \ef(x)\ef(y)\\
            \ef(x\ef(y)) & = & \ef(x)\ef(y)\\
            {}[x,\ef{[y,z]}] & = & 0\\
            {}[\ef(x),[\ef(y),z]] & = & 0.
        \end{eqnarray*}
        Then there is some \super{}algebra $\eA$ with \super{}trace $\estr$, such that $A$ and $\eA$ have the
        same multilinear identities with linear function $\ef$ and $\estr$ respectively.
    \end{cor}

    \subsection{Concluding Remarks}

    We have seen how the structure of the generalized Grassmann algebra can be used to
    generalize the notions of superalgebras and supertraces to arbitrary characteristics and rings. In a similar
    manner, one can define a Lie \super{}algebra:

    \begin{dfn}
        Let $\eL$ be a $C[\e]$-module with a \super{}algebra grading. Suppose that
        $\{\cdot,\cdot\}$ is a bi-linear form that respects the grading (if $a\in\eL_g,b\in\eL_h$ then $\{a,b\}\in\eL_{gh}$).
        Then $\eL$ will be called a Lie \super{}algebra if for all homogenous $x,y,z\in\eL$:
        \begin{enumerate}
            \item $\{x,y\}=-\exp(\e_x\e_y)\{y,x\}$,
            \item $\exp(\e_x\e_z)\{x,\{y,z\}\}+\exp(\e_y\e_x)\{y,\{z,x\}\}+\exp(\e_z\e_y)\{z,\{x,y\}\}=0$,
            \item\label{idL3} $\{x,\{x,x\}\}=0$.
        \end{enumerate}
    \end{dfn}

    Note that \ref{idL3} is superfluous when $3$ is invertible in $C$.
    This new object is obviously equivalent to an ordinary Lie superalgebra whenever $2$ is invertible. However,
    the interesting property of this definition is that it yields non-trivial behaviour in characteristic $2$,
    where (unlike ordinary Lie superalgebras) it does not degenerate to an ordinary Lie algebra.

    \smallskip

    In this paper we only considered \super{}theory from the point of view of PI-theory. In a similar manner, one can consider all of \super{}theory in characteristic $2$. The cost we pay for this
    is that since the grading is over an infinite group, we must consider infinite-dimensional objects; therefore,
    in order to replicate the study of finite dimensional objects, one should consider \super{}objects that are
    locally finite-dimensional, in the sense that their graded components are each finite dimensional and
    isomorphic to one another in a sufficiently strong sense (so infinite-dimensional behavior is not ``hidden"
    across multiple graded components).

        One hopes
        that this construction can be used to yield characteristic-free results over arbitrary rings, such as \Tref{thm:basis_of_identities_of_free_str}.

        \bibliographystyle{alpha}
        \bibliography{./bib/PI_bib}

\medskip

    \end{document}